\documentclass[11pt]{amsproc}%
\usepackage{amsmath}
\usepackage{amssymb}
\usepackage{amsthm}

\def\le {\leqslant}
\def\ge {\geqslant}

\usepackage{amssymb}
\usepackage{amsfonts}
\usepackage{amsmath}
\usepackage{graphicx}%
\RequirePackage{srcltx}

\usepackage{amssymb}
\RequirePackage{srcltx}

\providecommand{\U}[1]{\protect\rule{.1in}{.1in}}
\theoremstyle{plain}

\newtheorem{theorem}{Theorem}[section]
\newtheorem{lemma}[theorem]{Lemma}
\newtheorem{definition}[theorem]{Definition}

\newtheorem{remark}[theorem]{Remark}
\newtheorem{proposition}[theorem]{Proposition}

\numberwithin{equation}{section}

\newtheorem{corollary}[theorem]{Corollary}
\newtheorem{example}[theorem]{Example}

\usepackage{xcolor}
\begin{document}
\title[Sampling discretization  in Orlicz spaces]
{Sampling discretization  in Orlicz spaces}

\author{Egor Kosov}
\address{\noindent E. Kosov, Centre de Recerca Matem\`atica, Campus de Bellaterra, Edifici C 08193
Bellaterra (Barcelona), Spain. }
\email{kosoved09@gmail.com}

\author{Sergey Tikhonov}
\address{ S. Tikhonov \\
 ICREA, Pg.
Llu\'is Companys 23, 08010 Barcelona, Spain
\\
Centre de Recerca Matem\`atica
Campus de Bellaterra, Edifici~C, 08193 Bellaterra (Barcelona), Spain,
and Universitat Aut\`onoma de
Barcelona,
Facultat de Ci\`encies,  08193 Bellaterra (Barcelona), Spain.}
\email{stikhonov@crm.cat}


\begin{abstract}
We obtain new sampling discretization results in Orlicz norms on finite dimensional spaces.
As applications, we study sampling recovery problems, where the error of the recovery process is calculated with respect to different Orlicz norms.
In particular, we are interested in the recovery by linear methods in the norms close to $L^2$.
\end{abstract}

\maketitle


\section{Introduction}\label{sect1}
Let $\Omega\subset \mathbb{R}^d$ be a compact set and $\mu$ be a probability Borel measure on $\Omega$.
Let
$$
\|f\|_p:=\|f\|_{L^p(\mu)}:=\Bigl(\int_{\Omega}|f|^p\, d\mu\Bigr).
$$
For a continuous function $f\in C(\Omega),$ let
$
\|f\|_\infty:=\max\limits_{x\in \Omega}|f(x)|.
$

In this paper, we study the discretization problem for an
integral norm on $N$-dimensional subspaces
$X_N\subset C(\Omega)$, which can be formulated as follows:
{\it
How can we substitute the initial measure $\mu$ with a discrete measure $\nu$ supported on some 
  subset ${\bf x}:=\{x_1, \ldots, x_m\}\subset \Omega$ in such a way that the initial integral norm and the corresponding discrete norm are close to each other for elements from $X_N$?
}

The classical Marcinkiewicz discretization in $L^p$-norms
is a particular case of this problem.
We  say that, for a subspace $X_N$,  a Marcinkiewicz-type discretization theorem holds
with parameters $m\in\mathbb{N}$, $p\in[1, +\infty)$, and $C_2\ge C_1>0$ if
there is a set
${\bf x}:=\{x_1, \ldots, x_m\}\subset \Omega$ such that
$$
C_1\|f\|_p^p\le \frac{1}{m}\sum_{j=1}^{m}|f(x_j)|^p\le
C_2\|f\|_p^p\quad \forall f\in X_N.
$$
This problem goes back to the classical results of
Marcinkiewicz and Zygmund for trigonometric polynomials \cite{Zy},
but its comprehensive  study in the abstract framework has been started only very recently
(see, e.g., \cite{DKT}, \cite{DPSTT1}, \cite{DPSTT2}, \cite{Kos}, the surveys \cite{DPTT} and \cite{KKLT}, and the references therein).
The Marcinkiewicz discretization problem in an abstract framework appeared to
be related to the problem of tight embeddings of finite dimensional subspaces of $L^p[0, 1]$ into discrete
$\ell_p^m$
(see, e.g., \cite{BLM}, \cite{Sche87}, \cite{SZ01}, \cite{Ta2}, \cite{Ta3}, the survey \cite{JS82}, and the references therein).


\vskip .1in
\noindent
{\bf Marcinkiewicz-type  discretization for Orlicz norms.}
The main goal of this paper is
to study Marcinkiewicz-type sampling discretization for {\it integral Orlicz norms}
generated by a general ${\bf \Phi}$-function (see Definition \ref{Def01} below). As it is well known,
Orlicz spaces play a crucial role in various fields of analysis and PDEs.
There is a vast literature
on the general theory of Orlicz spaces;  we refer the reader to the classical monographs
\cite{KR60}, \cite{RR02}
and the recent book \cite{HH19}.

Our primary interest in the context of sampling discretization for Orlicz norms is related to the following observation.
It is known that, for $p\in[1, 2]$,
$O(N)$ points (up to some logarithmic  factor in $N$ for $p\in[1, 2)$)
are sufficient for an effective discretization
of $L^p$ norm on an $N$-dimensional subspace $X_N$, satisfying the Nikolskii-type inequality assumption
\begin{equation}\label{eq-Nik}
\|f\|_\infty\le \sqrt{BN}\|f\|_2\quad \forall f\in X_N.
\end{equation}
However, if $p>2$ and
$X_N$  satisfies \eqref{eq-Nik},
 one already needs at least $O(N^{p/2})$ points 
  (see {\bf D.20} in \cite{KKLT}).

Our first aim is to examine  the transition from the case $p=2$ to the case $p>2$ in terms of the  parameter $m$,
being the number of points sufficient for discretization.
One of our main results is that, for any $N$-dimensional subspace $X_N$ satisfying the assumption \eqref{eq-Nik},
$O(\Phi (\sqrt{N}))$ (up to logarithmic in $N$ factor) points are sufficient
for an effective discretization of the Orlicz norm, generated by a ${\bf \Phi}$-function $\Phi$ such that $ \Phi(t)t^{-2}$ is increasing.
In more detail, our first main result can be formulated as follows (see Theorems \ref{cor4.1} and \ref{t3.2} below).

\begin{theorem}\label{theo1.1}
Let $p\in [1, \infty)$, $\varepsilon\in(0, 1)$.
Let $\Phi$ be a ${\bf \Phi}$-function such that the function $t\mapsto \Phi(t)t^{-p}$
is increasing and assume that there is $q\in[p, \infty)$ such that the function $t\mapsto \Phi(t)t^{-q}$
is decreasing.
There is a constant $C_{\Phi, \varepsilon}>0$
depending only on $\Phi$ and $\varepsilon$ such that
for every $N$-dimensional subspace $X_N$, satisfying  condition \eqref{eq-Nik},
there exist
$$
m\le C_{\Phi, \varepsilon}\Phi\bigl((BN)^{\frac{1}{\min\{p, 2\}}}\bigr)(\log 2BN)^3
$$
and  a subset ${\bf x}:=\{x_1, \ldots, x_m\}\subset \Omega$
of cardinality $m$
such that
$$
(1-\varepsilon)\|f\|_\Phi\le \|f\|_{\Phi, {\bf x}}\le (1+\varepsilon)\|f\|_\Phi
\quad \forall f\in X_N.
$$
Here
$\|\cdot\|_\Phi$ is an integral Orlicz norm generated by the function $\Phi$
and $\|\cdot\|_{\Phi, {\bf x}}$ is its corresponding discrete counterpart for a discrete uniform measure
on the set ${\bf x}$.
\end{theorem}
Our approach is probabilistic in nature and relies on Theorem 1.2 from \cite{GMPT07},
which is a straightforward  corollary of the Talagrand's generic chaining method.
We will also use basic Dudley's entropy bound (see estimate \eqref{Dudley} below) to estimate gamma functionals appearing
from the application of the mentioned result from \cite{GMPT07}.

\vskip .1in
\noindent
{\bf One-sided discretization results.}
The second part of the paper studies discretization for Orlicz norms on arbitrary subspaces $X_N$,
not necessarily satisfying  assumption  \eqref{eq-Nik}.
For an arbitrary  $X_N$, under some technical assumptions
on the function $\Phi$, which  generates the Orlicz norm,
we derive one-sided discretization results that bound the integral Orlicz norm by a suitable discrete norm.
We consider two distinct cases: one where this discrete norm is generated by an Orlicz functional close to the original one,
and another where this norm is simply a discrete $L^2$ norm.
In the latter case, we generalize the results from \cite{KPUU24} as follows (see Corollary \ref{T-1sided-1} below).

\begin{theorem}\label{in fact corollary5.10}
Let $\Phi$ be ${\bf \Phi}$-functions such that
$\Phi(t)t^{-p}$ is increasing for some $p\ge 2$.
There are positive numbers $c_1$ and   $c_2:=c_2(\Phi, p)$
such that
for every $N$-dimensional subspace $X_N\subset C(\Omega)$,
there is a set ${\bf x}:=\{x_1, \ldots, x_m\}\subset \Omega$
of cardinality $m\le c_1N$ for which
$$
\|f\|_\Phi\le c_2\Bigl(\frac{\Phi\bigl(\sqrt{N}\bigr)}{N}\Bigr)^{1/p}\Bigl(\frac{1}{m}\sum_{j=1}^{m}|f(x_j)|^2\Bigr)^{1/2}
\quad \forall f\in X_N.
$$
\end{theorem}

This result and its more general version given in Theorem \ref{T-1sided} allow
for the  study of one-sided discretization in spaces close to $L^2$ and $L^\infty$.
As model cases,  we consider
$\displaystyle\Phi_{p, \alpha, \beta}(t)= t^p\frac{(\ln(e+t))^\alpha}{(\ln(e+t^{-1}))^\beta}$, $p\ge 1$, $\alpha, \beta\ge 0$
and $\Phi_q(t):= e^{t^q} - 1$. 
Then our results state that
there exists a set ${\bf x}:=\{x_1, \ldots, x_m\}\subset \Omega$ of cardinality
$$
m\le c_1N
$$
such that
$$
\|f\|_{\Phi_{p, \alpha, \beta}}\le c_2N^{\frac{1}{2}-\frac{1}{p}}(\log 4N)^{\frac{\alpha}{p}}
\Bigl(\frac{1}{m}\sum_{j=1}^{m}|f(x_j)|^2\Bigr)^{1/2} \quad \forall f\in X_N, \qquad p\ge 2.
$$
Similarly,
$$
\|f\|_{\Phi_q}\le c_2\frac{N^{1/2}}{(\log(N+1))^{1/q}} \Bigl(\frac{1}{m}\sum_{j=1}^{m}|f(x_j)|^2\Bigr)^{1/2}\quad \forall f\in X_N, \qquad q\ge 2.
$$
See Examples \ref{cor5.1} and \ref{cor5.2}.

\vskip .1in
\noindent
{\bf Sampling recovery by linear and non-linear methods.}
It has been recently discovered  that general one-sided discretization results as Theorem \ref{in fact corollary5.10} above
are linked with the problem of optimal sampling recovery (see \cite{DT22}, \cite{KPUU24}, \cite{LMT24}, \cite{Tem21}).
The problem of reconstruction of an unknown function $f$
defined on a domain $\Omega\subset \mathbb{R}^d$
from its samples at a finite set of points
${\bf x}=\{x_1, \ldots, x_m\}$ is an important problem of modern approximation theory.
We refer the interested reader to the following textbooks for the exposition of known results in the field:
\cite{Nov88}, \cite{TemBook}, \cite{TWW88}.

Let us recall the sampling recovery setting.
Let ${\bf F}\subset C(\Omega)$ and let us fix a
norm $\|\cdot\|$.
The sampling numbers of a function class ${\bf F}$ are given by
$$
\varrho_m({\bf F}, \|\cdot\|):=
\inf_{\substack{{\bf x}\subset \Omega\\ \# {\bf x}\le m}}
\inf_{ T_{{\bf x}} - \hbox{ linear}}
\sup\limits_{f\in {\bf F}}
\|f-T_{{\bf x}}(f(x_1), \ldots, f(x_m))\|,
$$
that is, the sampling numbers correspond to the uniformly optimal recovery of
a function from ${\bf F}$ by its sample in $m$ fixed points by linear methods of reconstruction.
Similarly, we define the modified sampling numbers
$$
\varrho_m^*({\bf F}, \|\cdot\|):=
\inf_{\substack{{\bf x}\subset \Omega\\ \# {\bf x}\le m}}
\inf_{X_N, N\le m}
\inf_{T_{{\bf x}}\colon \mathbb{C}^m\to X_N}
\sup\limits_{f\in {\bf F}}
\|f-T_{{\bf x}}(f(x_1), \ldots, f(x_m))\|.
$$
In other words,
considering $\varrho_m^*({\bf F}, \|\cdot\|)$, we also allow non-linear methods of reconstruction.
The behavior of sampling numbers
$\{\varrho_m({\bf F}, \|\cdot\|)\}_{m=1}^\infty$ and $\{\varrho_m^*({\bf F}, \|\cdot\|)\}_{m=1}^\infty$
are  well studied 
in the $L^2$-case, that is, when $\|\cdot\|=\|\cdot\|_{L^2(\mu)}$
(see \cite{DKU23}, \cite{JUV}, \cite{KSUW23}, \cite{KU21}, \cite{KU21-2}, \cite{NSU22}, \cite{Tem21},
\cite{WW01}).
In particular,
it was shown  that sampling numbers admit
sharp upper bounds in terms of the Kolmogorov widths.
More than that,  there are two positive constants $B$ and $b$ such that
for any compact subset ${\bf F}$ of $C(\Omega)$ one always has
\begin{equation}\label{eq-Tem}
\rho_{bN}({\bf F}, \|\cdot\|_{L^2(\mu)})\le
Bd_N({\bf F}, \|\cdot\|_\infty),
\end{equation}
where $\{d_N({\bf F}, \|\cdot\|_\infty)\}_{N=1}^\infty$ are the Kolmogorov
widths with respect to the uniform norm (see \cite{Tem21}).

The problem of recovery in $L^p$-norm, $p>2$, is far less studied.
Some general bounds for the sampling numbers
in the $L^p$-case, in terms of the classical Kolmogorov widths, have been obtained
in the recent paper \cite{KPUU24}. Apart from this result,
we  can only mention known results for modified sampling numbers (see, e.g., {\bf R.3} in \cite{KKLT})
and bounds involving some counterparts of the classical Kolmogorov widths (see Section 4 in \cite{DT22}).


Our goal here is to investigate the problem of sampling recovery in intermediate cases, specifically,
for the spaces that lie between $L^2$ and $L^p$,  $p>2$. In more detail, applying new
 discretization results
we derive new bounds for the sampling numbers in Orlicz norms.
These estimates generalize those from \cite{KPUU24} and recover \eqref{eq-Tem}.
Namely, we obtain the following result (see Theorem \ref{t7.1} below).
\begin{theorem}
Let $\Phi$ be a ${\bf \Phi}$-function such that
$\Phi(t)t^{-p}$ is increasing for some $p\ge 2$.
There exist a positive numerical constant $c\ge1$ and a number $C(\Phi, p)\ge 1$,
depending only on $\Phi$ and $p$, such that,
for any probability Borel measure $\mu$ on $\Omega$ and
for any function class ${\bf F}\subset C(\Omega)$,
one has
$$
\varrho_{cN}({\bf F}, \|\cdot\|_{L^\Phi(\mu)})
\le C(\Phi, p)\Bigl(\frac{\Phi\bigl(\sqrt{N}\bigr)}{N}\Bigr)^{1/p}d_N({\bf F}, \|\cdot\|_\infty).
$$
\end{theorem}
As above, we consider two model examples of Orlicz functions $\displaystyle\Phi_{p, \alpha, \beta}(\cdot)$
and $\Phi_q(\cdot)$. For them, our general results imply the estimates
$$
\varrho_{cN}({\bf F}, \|\cdot\|_{L^{\Phi_{p, \alpha, \beta}}(\mu)})
\le CN^{\frac{1}{2} - \frac{1}{p}}(\log 4N)^{\alpha/p}d_N({\bf F}, \|\cdot\|_\infty), \qquad p\ge 2,
$$
$$
\varrho_{cN}({\bf F}, \|\cdot\|_{L^{\Phi_q}(\mu)})
\le C\frac{N^{1/2}}{(\log(N+1))^{1/q}}d_N({\bf F}, \|\cdot\|_\infty), \qquad q\ge 2,
$$
where ${\bf F}\subset C(\Omega)$ is an arbitrary function class. See Examples \ref{ex7.3} and \ref{ex7.1} below.



\vskip .1in
\noindent
{\bf Structure of the paper and notation.}
The paper is organized as follows. In Section \ref{sect2}, we discuss the definitions and properties of
$\Phi$-functions and Orlicz spaces. Moreover, we emphasize  the importance of the Nikolskii condition and chaining technique in addressing  the discretization problem.
Section \ref{sect3} is devoted to the proof of Theorem~\ref{theo1.1}. Section \ref{sect4} studies possible relaxations of the conditions on
the ${\bf \Phi}$-function in Theorem~\ref{theo1.1} under which suitable discretization results are still valid.
Section \ref{sect6} is devoted to the study of unconditional (in terms of subspace)
discretization theorems. In their turn, these results are  used in Section \ref{sect7}
to provide new bounds for the sampling numbers.

Throughout the paper symbols $C, C_0, C_1, \ldots$ and $c, c_0, c_1, \ldots$
denote universal numerical constant, the value of which may vary from line to line. Similarly,
symbols $C(\alpha), C_0(\alpha), C_1(\alpha), \ldots$ and $c(\alpha), c_0(\alpha), c_1(\alpha), \ldots$
denote constants  that may depend only on the set of parameters $\alpha$ and the value of which may also vary from line to line. By $\log x$
we always denote the logarithm to the base $2$,
i.e., $\log x:=\log_2 x$.
Finally, we will intensively use the notation $X_N^\Phi$ to denote the unit (open) ball in the subspace $X_N$ with respect to the norm $\|\cdot\|_\Phi$, defined as
$$
X_N^\Phi:=\{f\in X_N\colon \|f\|_\Phi< 1\}
$$
and, in the case of the unit ball with respect to the $L^p$-norm,
$$
X_N^{p} :=\{f\in X_N\colon \|f\|_{p}< 1\}.
$$

\vspace{0.2cm}
\section{Preliminaries: Orlicz space and chaining bounds}\label{sect2}

\subsection{$\Phi$-functions}\label{ssect2-1}
We are partially following the notation from \cite{HH19}. The following definition was introduced by Bernstein in \cite{bernstein1949}; see also  \cite[Def. 2.1.1]{HH19}.

\begin{definition}
Let $t_*\in\mathbb{R}$.
A function $\varphi\colon (t_*, \infty)\to \mathbb{R}$ is called almost increasing (respectively, decreasing) on $(t_*, \infty)$ with the constant
$a\ge 1$
if $\varphi(s)\le a\varphi(t)$ ($\varphi(t)\le a\varphi(s)$) for all $t_*<s<t$.
If $a=1$, then $\varphi$ is called increasing (decreasing) on $(t_*, \infty)$.
\end{definition}

\begin{remark}
{\rm
We note that $\varphi\colon (t_*, \infty)\to \mathbb{R}$ is almost increasing (respectively, decreasing)
if and only if there is an increasing (decreasing) function $\tilde{\varphi}$ and a constant $a\ge 1$
such that
\begin{equation}\label{tilde-phi}
\tilde{\varphi}(t)\le \varphi(t)\le a\tilde{\varphi}(t)\quad \forall t\in (t_*, \infty).
\end{equation}
It is enough to take $\tilde{\varphi}(t):=a^{-1}\sup\limits_{t_*< s\le t}\varphi(s)$ if $\varphi$ is almost increasing  and
$\tilde{\varphi}(t):=\inf\limits_{t_*< s\le t}\varphi(t)$ for almost decreasing $\varphi$.
Moreover, for $t_*<s<t$, condition \eqref{tilde-phi} implies $\varphi(s)\le a\varphi(t)$ if $\tilde{\varphi}$ is increasing
 and
$\varphi(t)\le  a\varphi(s)$ if $\tilde{\varphi}$ is decreasing.
}
\end{remark}

\begin{definition}
Let $\Phi\colon(0, +\infty)\to \mathbb{R}$ and $p, q\ge1$.
We write

$(i)$ $\Phi\in {\rm(Inc)}_p(\infty)$ $($respectively, $\Phi\in {\rm(Inc)}_p)$
if $t\mapsto \Phi(t)t^{-p}$ is increasing on
$(t_*, \infty)$ for some $t_*\ge 0$ $($on $(0, \infty))$;

$(ii)$ $\Phi\in {\rm (aInc)}_p(\infty)$
$($respectively, $\Phi\in {\rm (aInc)}_p)$ if $t\mapsto \Phi(t)t^{-p}$ is almost increasing with some constant $a_\Phi(p)$
on
$(t_*, \infty)$ for some $t_*\ge 0$ $($on $(0, \infty))$;

$(iii)$ $\Phi\in {\rm(Dec)}_q(\infty)$
$($respectively, $\Phi\in {\rm(Dec)}_q)$ if $t\mapsto \Phi(t)t^{-q}$ is decreasing
on
$(t_*, \infty)$ for some $t_*\ge 0$ $($on $(0, \infty))$;

$(iv)$ $\Phi\in {\rm(aDec)}_q(\infty)$
$($respectively, $\Phi\in {\rm(aDec)}_q)$ if $t\mapsto \Phi(t)t^{-q}$ is almost decreasing with some constant $b_\Phi(q)$
on
$(t_*, \infty)$ for some $t_*\ge 0$ $($on $(0, \infty))$.
\end{definition}

Set
\begin{eqnarray*}
&&{\rm(Dec)} := \bigcup\limits_{q\ge 1} {\rm(Dec)}_q,\qquad\qquad\;\,\quad
{\rm(aDec)} := \bigcup_{q\ge 1} {\rm(aDec)}_q,
\\
&&{\rm(Dec)}(\infty) := \bigcup_{q\ge 1} {\rm(Dec)}_q(\infty),\qquad
{\rm(aDec)}(\infty) := \bigcup_{q\ge 1} {\rm(aDec)}_q(\infty).
\end{eqnarray*}
The defined function classes are closely related to the concept of regular variation, see the monograph \cite{Bingham}
for more detail.

\begin{remark}\label{rem-diff}
{\rm
For a differentiable function $\Phi\colon (0, +\infty)\to (0, +\infty)$,
the condition $\Phi\in{\rm(Inc)}_p$ (or ${\rm(Dec)}_q$,  ${\rm(Inc)}_p(\infty)$, ${\rm(Dec)}_q(\infty)$)
can be simply verified by studying the function
$\varphi(t):=\frac{t\Phi'(t)}{\Phi(t)}$.
Indeed, since
$
(\Phi(t)t^{-p})' =
\Phi(t)t^{-p-1}\bigl(\varphi(t)-p\bigr)$ for $p>0$,
the condition  $\inf\limits_{t>t_*}\varphi(t)\ge p$ yields that
$t\mapsto \Phi(t)t^{-p}$ is increasing on $(t_*, \infty)$.
Similarly, if $\sup\limits_{t>t_*}\varphi(t)\le q$ then
$t\mapsto \Phi(t)t^{-q}$ is decreasing on $(t_*, \infty)$.
}
\end{remark}

\begin{definition}
[see  \cite{HH19}]\label{Def01}
Let $\Phi\colon [0, +\infty)\to [0, +\infty)$
be an increasing function such that $\Phi(0)=0$,
$\lim\limits_{t\to+0}\Phi(t) = 0$, and $\lim\limits_{t\to +\infty}\Phi(t)=+\infty$. Such $\Phi$ is called a ${\bf \Phi}$-prefunction.
\\
A ${\bf \Phi}$-prefunction $\Phi$ is called

\begin{itemize}
  \item[$(i)$] ${\bf \Phi}$-function if $\Phi\in {\rm (aInc)}_1$
(written as $\Phi\in {\bf \Phi}_w$);
  \item[$(ii)$] convex ${\bf \Phi}$-function if it is a convex function
(written as $\Phi\in {\bf \Phi}_c$).
\end{itemize}
\end{definition}

We point out that, for $\Phi\in {\rm (aInc)}_p$ and  $\lambda\ge 1$,
we have
\begin{equation}\label{eq-est-p}
\Phi\bigl(a_\Phi(p)^{-1/p}\lambda^{-1/p}t\bigr)\le \lambda^{-1}\Phi(t), \qquad t\in[0, +\infty).
\end{equation}
In particular,
\begin{equation}\label{eq-est}
\Phi\bigl(a_\Phi(1)^{-1}\lambda^{-1}t\bigr)\le \lambda^{-1}\Phi(t), \qquad t\in[0, +\infty)
\end{equation}
for
$\Phi\in {\bf \Phi}_w$ and  $\lambda\ge 1$.

For a ${\bf \Phi}$-function $\Phi\in {\rm (aInc)}_p$,  set
\begin{equation}\label{eq-R}
R_\Phi(p):=\Bigl(1+\frac{a_\Phi(p)}{\Phi(1)}\Bigr)^{1/p}.
\end{equation}


We say that two functions $\Phi, \Psi\colon [0, \infty)\to [0, \infty)$
are equivalent (and write $\Phi\simeq \Psi$) if
there is a number $L\ge 1$ such that
$$
\Phi(t/L)\le \Psi(t)\le \Phi(Lt)\quad \forall t\ge0.
$$
It worth noting that the classes
${\rm(aInc)}_p$ and ${\rm(aDec)}_q$ are invariant
under the equivalence of ${\bf \Phi}$-functions
(see \cite[L. 2.1.9]{HH19}).
Another important observation regarding a weak ${\bf \Phi}$-function
 is that it can always be  upgraded to a convex ${\bf \Phi}$-function (see \cite[L. 2.2.1]{HH19}).

\begin{lemma}
\label{lem-equiv}
Let $p\in[1, \infty)$.
If $\Phi\in {\bf \Phi}_w$ and $\Phi\in {\rm (aInc)}_p$,
then there exists $\Psi\in {\bf \Phi}_c$ equivalent to $\Phi$
such that $\Psi^{1/p}$ is convex.
In particular, $\Psi\in {\rm (Inc)}_p$. 
\end{lemma}

\subsection{Orlicz space}\label{ssect2-2}
Let $\mu$ be a probability measure on some set $\Omega$.
For $\Phi\in {\bf \Phi}_w$ we set
$$
\rho_\Phi(f):= \int_{\Omega}\Phi(|f(x)|)\, \mu(dx).
$$
Define
$$
L^\Phi(\mu):=\{f\colon \exists \lambda>0\colon \rho_\Phi(f/\lambda)<\infty\}.
$$
and, for $f\in L^\Phi(\mu)$, let the Luxembourg functional be given by
$$
\|f\|_{L^\Phi(\mu)}=\|f\|_\Phi:=
\inf\{\lambda>0\colon \rho_\Phi(f/\lambda)\le 1\}.
$$
It is known (see \cite[Ch. 3]{HH19})
that $\|\cdot\|_\Phi$ is a norm if $\Phi\in {\bf \Phi}_c$
and $\|\cdot\|_\Phi$ is a quasinorm if $\Phi\in {\bf \Phi}_w$, that is,
\begin{equation}\label{eq-triangle}
\|f+g\|_\Phi\le C_\Phi(\|f\|_\Phi + \|g\|_\Phi),\qquad C_\Phi>0,
\end{equation}
for every $f, g\in L^\Phi(\mu)$.
In particular, $\|\cdot\|_\Phi$
always is a positively homogeneous functional.
Moreover,
 for any $\Phi\in {\bf \Phi}_w$ one has
\begin{equation}\label{eq-ball}
\|f\|_\Phi<1\Rightarrow \rho_\Phi(f)\le 1
\Rightarrow \|f\|_\Phi\le 1.
\end{equation}


In the discrete setting, i.e., when $\Omega:={\bf x}=\{x_1, \ldots, x_m\}$
and $\mu=\sum\limits_{j=1}^{m}\lambda_j\delta_{x_j}$ for some set of weights $\lambda:=\{\lambda_1, \ldots, \lambda_m\}$,
we will use the notation
$$
\|f\|_{\Phi, {\bf x}, \lambda}:= \|f\|_{L^\Phi(\sum\limits_{j=1}^{m}\lambda_j\delta_{x_j})}.
$$
In the case $\lambda_1=\ldots=\lambda_m$, we simply write
$\|f\|_{\Phi, {\bf x}}$.
For $L^p(\mu)$-norms, $p\in[1, \infty)$, as usual, we define
$$
\|f\|_{L^p(\mu)}=\|f\|_p:=\Bigl(\int_{\Omega}|f|^p\, d\mu\Bigr)^{1/p}.
$$

\begin{lemma}\label{lem0}
Let $q\in[1, +\infty)$ and let $\Phi$ be a ${\bf \Phi}$-function such that $\Phi\in {\rm(Dec)}_q$.
Then
$$
|\Phi(u) - \Phi(v)|\le q |u - v|
\Bigl(\frac{\Phi(u)}{u} + \frac{\Phi(v)}{v}\Bigr)
\quad \forall u, v\in[0, +\infty).
$$
\end{lemma}

\begin{proof}
Assume that $u>v$.
Then
$$
\frac{\Phi(u) - \Phi(v)}{u-v} =
\Phi(u)u^{-q}\frac{u^q}{u-v} - \Phi(v)v^{-q}\frac{v^q}{u-v}
\le\Phi(u)u^{-q}\frac{u^q-v^q}{u-v}
\le
q\Phi(u)u^{-1},
$$
which implies the required estimate.
\end{proof}

\begin{lemma}\label{lem0.5}
Let $p\in[1, +\infty)$ and let $\Phi$ be a ${\bf \Phi}$-function such that $\Phi\in {\rm (aInc)}_p$.
Then, for any function $h\in C(\Omega)$, one has
$$
\|h\|_{L^p(\mu)}\le R_\Phi(p)\|h\|_{L^\Phi(\mu)},
$$
where $R_\Phi(p):=\bigl(1+\frac{a_\Phi(p)}{\Phi(1)}\bigr)^{1/p}$, cf.  \eqref{eq-R}.
\end{lemma}

\begin{proof}
Let $h\in C(\Omega)$ with $\|h\|_{L^\Phi(\mu)}<1$.
We note that
$a_\Phi(p)\Phi(t)\ge \Phi(1)t^p$ for $t\ge 1$. Then
there holds
$$
\int_{\Omega}|h|^p\, d\mu\le
\mu(|h|\le 1) + \frac{a_\Phi(p)}{\Phi(1)}
\int_{|h|\ge1}\Phi(|h|)\, d\mu
\le 1+\frac{a_\Phi(p)}{\Phi(1)}= \bigl(R_\Phi(p)\bigr)^p,
$$
where we have used the first implication in \eqref{eq-ball}.
\end{proof}

\begin{lemma}\label{lem5}
Let $\Phi\in{\bf \Phi}_w\cap {\rm (aDec)}_q$, $q\in[1, \infty)$,
and $\varepsilon\in(0, (b_\Phi(q))^{-1})$.
Suppose that
 $\mu$ and $\nu$ are two
probability measures on $\Omega$ and
 $X$ is a subspace in $C(\Omega)$.
Then the inequality
$$
\sup\limits_{\substack{f\in X\\ \|f\|_{L^\Phi(\mu)}< 1}}
\Bigl|\int_\Omega\Phi(|f|)\, d\nu - \int_\Omega\Phi(|f|)\, d\mu\Bigr|
\le \varepsilon
$$
implies
\begin{equation}\label{vsp}
a_\Phi(1)^{-1}((b_\Phi(q))^{-1}-\varepsilon)\|f\|_{L^\Phi(\mu)}\le \|f\|_{L^\Phi(\nu)}
\le a_\Phi(1)(1+\varepsilon)\|f\|_{L^\Phi(\mu)}\quad \forall f\in X.
\end{equation}
\end{lemma}

\begin{proof}
First, assume that $\|f\|_{L^\Phi(\mu)} < 1$.
Then by \eqref{eq-ball}, $\rho_\Phi(f)\le 1$.
Since $\Phi\in {\rm (aInc)}_1$, by \eqref{eq-est}, we have
$$
\int_\Omega\Phi\bigl(a_\Phi(1)^{-1}(1+\varepsilon)^{-1}|f|\bigr)\, d\nu
\le (1+\varepsilon)^{-1}\int_\Omega\Phi(|f|)\, d\nu
\le (1+\varepsilon)^{-1}\Bigl(\int_\Omega\Phi(|f|)\, d\mu + \varepsilon\Bigr)\le 1,
$$
which gives  $\|f\|_{L^\Phi(\nu)}\le a_\Phi(1)(1+\varepsilon)$.
The latter  yields
$$
\|f\|_{L^\Phi(\nu)}\le a_\Phi(1)(1+\varepsilon)\|f\|_{L^\Phi(\mu)}\quad \forall f\in X.
$$

Second, we suppose  that $1> \|f\|_{L^\Phi(\mu)} > \delta$ for some $\delta\in((b_\Phi(q)\varepsilon)^{1/q}, 1)$.
We note that $\rho_\Phi(\delta^{-1}f)> 1$ due to \eqref{eq-ball}.
Thus, in light of  \eqref{eq-est}, we have
\begin{align*}
\int_\Omega\Phi\bigl(a_\Phi(1)(1-b_\Phi(q)\delta^{-q}\varepsilon)^{-1}b_\Phi(q)&\delta^{-q}|f|\bigr)\, d\nu
\ge (1-b_\Phi(q)\delta^{-q}\varepsilon)^{-1}b_\Phi(q)\delta^{-q}\int_\Omega\Phi(|f|)\, d\nu
\\&\ge (1-b_\Phi(q)\delta^{-q}\varepsilon)^{-1}b_\Phi(q)\delta^{-q}
\Bigl(\int_\Omega\Phi(|f|)\, d\mu - \varepsilon\Bigr)
\\&\ge
(1-b_\Phi(q)\delta^{-q}\varepsilon)^{-1}
\Bigl(\int_\Omega\Phi(\delta^{-1}|f|)\, d\mu - b_\Phi(q)\delta^{-q}\varepsilon\Bigr)
> 
1,
\end{align*}
that is,  $\|f\|_{L^\Phi(\nu)}\ge a_\Phi(1)^{-1}((b_\Phi(q))^{-1}\delta^{q}-\varepsilon)$.
We apply this bound  for $\frac{1+\delta}{2}\frac{f}{\|f\|_{L^\Phi(\mu)}}$ to get
$$
\|f\|_{L^\Phi(\nu)}\ge a_\Phi(1)^{-1}((b_\Phi(q))^{-1}\delta^{q}-\varepsilon)\frac{2}{1+\delta}\|f\|_{L^\Phi(\mu)}.
$$
Taking  the limit as $\delta\to 1-$ completes the proof of the left-hand side estimate in \eqref{vsp}.
\end{proof}

\subsection{Nikolskii's inequality}\label{ssect2-3}
A key assumption to study various problems of discretization is the Nikolskii inequality
(see, e.g., \cite{devore1}, \cite{devore2}).
Let $X_N$ be an $N$-dimensional subspace of $C(\Omega)$, where a compact set $\Omega$
is endowed with a probability measure $\mu$.
We say that the subspace $X_N$ satisfies the Nikolskii inequality for the pair $(p, q)$ with the constant $H=H(X_N,p,q)$
if
$$
\|f\|_q\le H\|f\|_p\quad \forall f\in X_N.
$$
In this case we write $X_N \in NI_{p, q}(H)$.
The case $p=2$ and $q=\infty$ is particularly interesting due to the following observation. 
\begin{remark}
{\rm
It is easy to see
that for any orthonormal basis $\{u_1, \ldots, u_N\}$ of $X_N$ in $L^2(\mu)$,
one has
$
\displaystyle \sum_{k=1}^{N}|u_k(x)|^2 =
\sup\limits_{\substack{f\in X_N\\ \|f\|_2\le 1}}|f(x)|^2.
$
Thus,
$$
X_N \in NI_{2, \infty}(H)\qquad\mbox{ if and only if}\qquad
\sup\limits_{x\in \Omega}\sum_{k=1}^{N}|u_k(x)|^2\le H^2.
$$
We note that one always has $H=H(X_N, 2, \infty)\ge\sqrt{N}\ge 1$.
}
\end{remark}

An important example of the space $X_n$ satisfying Nikolskii's inequality is the set of trigonometric polynomials
$$
\mathcal{T}(Q):=\Bigl\{f(x)=\sum_{k\in Q}c_ke^{i\langle k, x\rangle}\colon c_k\in\mathbb{C}\Bigr\}
$$
with $Q\subset \mathbb{Z}^d$, $\Omega = [0, 2\pi)^d$, $d\mu = \frac{1}{(2\pi)^d}I_\Omega d\lambda$. One has
$\mathcal{T}(Q)\in  NI_{2, \infty}(|Q|^{1/2})$, see \cite{devore1}.

\subsection{Entropy numbers and chaining bound}\label{ssect2-4}
\begin{definition}\label{D-ent}
Let $(F, d)$ be a
metric space.
The entropy numbers $\{e_n(F, d)\}_{n=0}^\infty$ are defined as follows:
$$
e_n(F, d):=
\inf\Bigl\{\varepsilon>0\colon \exists f_1,\ldots, f_{N_n}\in F\colon
F\subset \bigcup\limits_{j=1}^{N_n}B_\varepsilon(f_j)\Bigr\},
$$
where $N_n=2^{2^n}$ for $n\ge 1$ and $N_0=1$
and
$B_\varepsilon(f):=\{g\colon d(f,g)<\varepsilon\}$.
\end{definition}
If the metric $d$ is induced by a norm $\|\cdot\|$, we will
also use the notation $e_n(F,\|\cdot\|)$ in place of $e_n(F, d)$.

\begin{definition}
An admissible sequence of a set $F$ is an increasing sequence $(\mathcal{F}_n)$ of partitions of $F$
such that $|\mathcal{F}_n|\le 2^{2^n}$ for all $n\ge1$ and $|\mathcal{F}_0|=1$.
For $f\in F$, let $F_n(f)$ denote the unique element of $\mathcal{F}_n$ that contains $f$.
\end{definition}

\begin{definition}
Let $(F, d)$ be a metric space.
Let
$$
\gamma_2(F,d):=
\inf\sup_{f\in F}\sum\limits_{n=0}^\infty 2^{n/2} \mathrm{diam}\bigl(F_n(f)\bigr),
$$
where $\mathrm{diam}(G):=\sup\limits_{f,g\in G}d(f,g)$
and the infimum is taken over all admissible sequences of $F$.
\end{definition}
The quantity $\gamma_2(F,d)$ is called the chaining functional.
If the metric $d$ is induced by a norm $\|\cdot\|$, we will
use the notation $\gamma_2(F, \|\cdot\|)$ in place of $\gamma_2(F, d)$.


The classical Dudley's entropy bound states that (see \cite[Prop. 2.2.10]{Tal} and  the preceding discussion)
\begin{equation}\label{Dudley}
\gamma_2(F, d)\le \sum\limits_{n=0}^\infty 2^{n/2} e_n(F,d).
\end{equation}

We will use the following theorem from  \cite{GMPT07} (see Theorem 1.2 there),
which is a combination of the chaining bound (see \cite{Tal}) and the Gin\'e--Zinn symmetrization argument.

\begin{theorem}\label{GMPT}
There is a numerical constant $c>0$ such that for any
i.i.d. random vectors $x_1, \ldots, x_m$ with the distribution $\mu$ on the set $\Omega$, one has
$$
\mathbb{E}\Bigl(\sup\limits_{g\in G}
\Bigl|\frac{1}{m}\sum\limits_{j=1}^m|g(x_j)|^2 - \int_\Omega|g|^2\, d\mu\Bigr|\Bigr)
\le
c\Bigl(A + A^{1/2}\Bigl(\sup\limits_{g\in G}\int_{\Omega}|g|^2\, d\mu\Bigr)^{1/2} \Bigr),
$$
where
$$
A=\frac{1}{m}\mathbb{E}\bigl(\gamma_2^2(G, \|\cdot\|_{\infty, {\mathbf{x}}})\bigr)
$$
and  $\|g\|_{\infty, {\mathbf{x}}} := \max\limits_{1\le j\le m}|g(x_j)|$.
\end{theorem}



We will also use the following useful property of the entropy numbers.

\begin{lemma}\label{lem4}
Let $a, b>0$.
Let $X_N$ be an $N$-dimensional space endowed with a norm $\|\cdot\|$ and
let $F\subset X_N$.
Then there is a number $C(a, b)>0$ such that
$$
\sum\limits_{n > [\log N]}\bigl(2^{a n} e_n(F, \|\cdot\|)\bigr)^b
\le C(a, b) \sum\limits_{n \le [\log N]}\bigl(2^{a n} e_n(F, \|\cdot\|)\bigr)^b.
$$
\end{lemma}

\begin{proof}
Let $n_0=[\log N]$.
It is known (see  (7.1.6)  and Corollary 7.2.2 in \cite{TemBook}) that
\begin{equation*}\label{eq1.1}
e_n(F,\|\cdot\|)\le 3\cdot 2^{2^{n_0}/N} e_{n_0}(F,\|\cdot\|) 2^{-2^n/N}\quad
\forall n>n_0.
\end{equation*}
Thus,
$$
\sum\limits_{n > [\log N]}(2^{a n} e_n(F, \|\cdot\|))^b
\le
6^b\bigl(e_{n_0}(F,\|\cdot\|)\bigr)^b\sum\limits_{n > [\log N]}\bigl(2^{a n} 2^{-2^n/N}\bigr)^b.
$$
We note that
\begin{eqnarray*}
\sum\limits_{n\ge \log N}\bigl(2^{a n}2^{-2^n/N}\bigr)^b
&\le&
2\max(2^{ab-1}, 1)\int_{0}^{\infty}x^{ab-1}2^{-xb/N}\, dx
\\
&=&
2\max(2^{ab-1}, 1)\bigl(\tfrac{N}{b\ln 2}\bigr)^{ab}
\int_{0}^{\infty}y^{ab-1}e^{-y}\, dy
=c(a, b)N^{ab}.
\end{eqnarray*}
Thus,
\begin{eqnarray*}
\sum\limits_{n > [\log N]}(2^{a n} e_n(F, \|\cdot\|))^b
&\le& 6^b c(a, b) N^{ab}\bigl(e_{n_0}(F,\|\cdot\|)\bigr)^b
\le 6^b c(a, b) 2^{ab} \bigl(2^{an_0}e_{n_0}(F,\|\cdot\|)\bigr)^b
\\
&\le& 6^b c(a, b) 2^{ab}
\sum\limits_{n \le [\log N]}\bigl(2^{a n} e_n(F, \|\cdot\|)\bigr)^b.
\end{eqnarray*}
The proof is now complete.
\end{proof}

\vspace{0.2cm}
\section{Discretization with equal weights for $\Phi\in {\rm (aInc)}_p\cap{\rm (Dec)}$}\label{sect3}

\subsection{Main theorems and discussion}\label{ssect3-1}

This section is devoted to the proof of the following two theorems for
${\bf \Phi}$-function satisfying the condition ${\rm (aInc)}_p\cap{\rm (Dec)}$.

\begin{theorem}\label{cor4.1}
Let $p\in (1, \infty)$.
Let $\Phi$ be a ${\bf \Phi}$-function such that $\Phi\in {\rm (aInc)}_p\cap{\rm (Dec)}$.
There exists a constant $C_{\Phi, p}>0$,
depending only on $p$ and $\Phi$, such that for every $N\ge 1$, $H\ge 1$,
for every $N$-dimensional subspace $X_N\in NI_{2, \infty}(H)$,
for every $\varepsilon\in(0, 1/2]$, and for every
$$
m\ge C_{\Phi, p}
\varepsilon^{-2}(\log\varepsilon^{-1})^{\frac{\min\{p, 2\}}{2}}
\Phi\bigl(H^{\frac{2}{\min\{p, 2\}}}\bigr)(\log 2N)^2\log 2H^2,
$$
there is a set ${\bf x}:=\{x_1, \ldots, x_m\}\subset \Omega$
of cardinality $m$
such that
$$
a_\Phi(1)^{-1}(1-\varepsilon)\|f\|_\Phi\le \|f\|_{\Phi, {\bf x}}\le a_\Phi(1)(1+\varepsilon)\|f\|_\Phi
\quad \forall f\in X_N.
$$
\end{theorem}

In the case $p=1$ we will prove the following counterpart of the above theorem.

\begin{theorem}\label{t3.2}
Let $\Phi\in{\bf \Phi}_w$ be such that $\Phi\in {\rm (Dec)}$.
There exists a constant $C_{\Phi}>0$, depending only on $\Phi$ such that
for every $N\ge 1$, $H\ge1$, for every $N$-dimensional subspace $X_N\in NI_{2, \infty}(H)$, and
for every $\varepsilon\in(0, 1/2]$, there are $m$ points $x_1, \ldots, x_m\in\Omega$ with
$$
m\le C_{\Phi} \varepsilon^{-2} (\log \varepsilon^{-1}) \Phi\bigl(H^2\bigr) (\log 2N)^2 \log 2H^2
$$
such that
$$
a_\Phi(1)^{-2}(1-\varepsilon)\|f\|_\Phi\le \|f\|_{\Phi, {\bf x}}\le a_\Phi(1)^2(1+\varepsilon)\|f\|_\Phi
\quad \forall f\in X_N.
$$
\end{theorem}

The proofs of Theorems \ref{cor4.1} and \ref{t3.2} are
divided into several steps and contained in  subsections 3.2-3.5.

Let us illustrate the use of Theorems \ref{cor4.1} and \ref{t3.2} for the specific ${\bf \Phi}$-function.
Set
\begin{equation}\label{eq-Phi}
\Phi_{p, \alpha, \beta}(t):= t^p\frac{(\ln(e+t))^\alpha}{(\ln(e+t^{-1}))^\beta},\quad
p\ge 1, \alpha, \beta\ge 0.
\end{equation}
 It is easy to see that
 for $t>0$, one has
\begin{equation*}\label{eq-ex}
\frac{t\Phi_{p, \alpha, \beta}'(t)}{\Phi_{p, \alpha, \beta}(t)}
=
p + \frac{\alpha t}{e+t}\cdot\frac{1}{\ln(e+t)} +
\frac{\beta t^{-1}}{e+t^{-1}}\cdot\frac{1}{\ln(e+t^{-1})}.
\end{equation*}
Thus,
$$
p\le \frac{t\Phi_{p, \alpha, \beta}'(t)}{\Phi_{p, \alpha, \beta}(t)}\le p+\alpha+\beta \quad \hbox{ for }\quad t>0.
$$
{In light of Theorem \ref{cor4.1}, in order to obtain the discretization estimates between $\Phi_{p, \alpha, \beta}$ and  $\Phi_{p, \alpha, \beta}, {\bf x}$, it is sufficient to choose 

$$
m\ge C_{p, \alpha, \beta}
\varepsilon^{-2}(\log\varepsilon^{-1})^{\frac{\min\{p, 2\}}{2}}
\Phi_{p, \alpha, \beta}\bigl(H^{\frac{2}{\min\{p, 2\}}}\bigr)(\log 2N)^2\log 2H^2.
$$
Noting that
$$
\Phi_{p, \alpha, \beta}\bigl(H^{\frac{2}{\min\{p, 2\}}}\bigr)\le
H^{\frac{2p}{\min\{p, 2\}}}\bigl(\ln(e+H^{\frac{2}{\min\{p, 2\}}})\bigr)^\alpha\le
c(p, \alpha) H^{\frac{2p}{\min\{p, 2\}}}(\log2H^2)^\alpha,
$$
we arrive at the following result.
}

\begin{example}\label{ex4.1}
{\rm
Let $p\in (1, \infty)$, $\alpha, \beta\ge0$, $N\ge 1$, $B\ge 1$.
Let $\Phi_{p, \alpha, \beta} := t^p\frac{(\ln(e+t))^\alpha}{(\ln(e+t^{-1}))^\beta}$.
There exists a constant $C_{p, \alpha, \beta}>0$,
depending only on $p$, $\alpha$, and $\beta$, such that
for every $N$-dimensional subspace $X_N\in NI_{2, \infty}(\sqrt{BN})$,
for every $\varepsilon\in(0, 1/2]$, and for every
$$
m\ge C_{p, \alpha, \beta}
\varepsilon^{-2}(\log\varepsilon^{-1})^{\frac{\min\{p, 2\}}{2}}
(BN)^{\frac{p}{\min\{p, 2\}}}(\log 2BN)^{\alpha+1}(\log 2N)^2,
$$
there is a subset ${\bf x}:=\{x_1, \ldots, x_m\}\subset \Omega$
of cardinality $m$
such that
$$
(1-\varepsilon)\|f\|_{\Phi_{p, \alpha, \beta}}\le \|f\|_{\Phi_{p, \alpha, \beta}, {\bf x}}
\le (1+\varepsilon)\|f\|_{\Phi_{p, \alpha, \beta}} \quad \forall f\in X_N.
$$
}
\end{example}

Similarly, Theorem \ref{t3.2} implies the discretization for spaces close to $L_1$.

\begin{example}{\rm
Let $\alpha, \beta\ge0$, $N\ge 1$, $B\ge 1$, and
let $\Phi_{1, \alpha, \beta}:=t\frac{(\ln(e+t))^\alpha}{(\ln(e+t^{-1}))^\beta}$.
There exists a constant $C_{\alpha, \beta}>0$,
depending only on the parameters $\alpha$ and $\beta$, such that,
for every $N$-dimensional subspace $X_N\in NI_{2, \infty}(\sqrt{BN})$ and
for every $\varepsilon\in(0, 1/2]$, there are
$$
m\le C_{\alpha, \beta}
\varepsilon^{-2}(\log\varepsilon^{-1}) BN(\log 2BN)^{\alpha+1}(\log 2N)^2.
$$
and  a subset ${\bf x}:=\{x_1, \ldots, x_m\}\subset \Omega$ of cardinality $m$
such that
$$
(1-\varepsilon)\|f\|_{\Phi_{1, \alpha, \beta}}\le \|f\|_{\Phi_{1, \alpha, \beta}, {\bf x}}\le (1+\varepsilon)\|f\|_{\Phi_{1, \alpha, \beta}}
\quad \forall f\in X_N.
$$
}
\end{example}

\begin{remark}\label{rem2.1}
{\rm
The proof of Theorem \ref{cor4.1} is probabilistic in its nature and it actually implies that,
under the assumptions of the theorem, for $\varepsilon\in(0, 1/2]$, $\delta\in(0, 1]$, and
$$
m\ge C_{\Phi, p}(\varepsilon\delta)^{-2}(\log(\varepsilon\delta)^{-1})^{\frac{\min\{p, 2\}}{2}}
\Phi\bigl(H^\frac{2}{\min\{p, 2\}}\bigr)(\log 2N)^2\log 2H^2,
$$
we have
$$
P\Bigl(a_\Phi(1)^{-1}(1-\varepsilon)\|f\|_\Phi\le \|f\|_{\Phi, {\bf x}}\le a_\Phi(1)(1+\varepsilon)\|f\|_\Phi
\quad \forall f\in X_N\Bigr)\ge 1-\delta/2,
$$
where the points ${\bf x}=(x_1, \ldots, x_m)$ are choosing
independently distributed according to the measure $\mu$.
}
\end{remark}

\subsection{Proof of Theorems \ref{cor4.1} and \ref{t3.2}: the key step}\label{ssect3-2}

To prove Theorems \ref{cor4.1} and \ref{t3.2}, we apply Theorem \ref{GMPT} with the set
$$
G:=\{g:=(\Phi(|f|))^{1/2}\colon f\in X_N^\Phi\}.
$$
Thus,
\begin{eqnarray*}
\mathbb{E}\Bigl(\sup\limits_{f\in X_N^\Phi}
\Bigl|\frac{1}{m}\sum\limits_{j=1}^m\Phi(|f(x_j)|) - \int_\Omega\Phi(|f|)\, d\mu\Bigr|
\Bigr)
&\le&
\mathbb{E}\Bigl(\sup\limits_{g\in G}
\Bigl|\frac{1}{m}\sum\limits_{j=1}^m|g(x_j)|^2 - \int_\Omega|g|^2\, d\mu\Bigr|\Bigr)
\\
\le
c\Bigl(A + A^{1/2}\Bigl(\sup\limits_{g\in G}\int_{\Omega}|g|^2\, d\mu\Bigr)^{1/2} \Bigr)
&\le&
c\Bigl(A + A^{1/2}\Bigl(\sup\limits_{f\in X_N^\Phi}\int_{\Omega}\Phi(|f|)\, d\mu\Bigr)^{1/2} \Bigr),
\end{eqnarray*}
where
$$
A=\frac{1}{m}\mathbb{E}\bigl(\gamma_{2, 1}^2(G, \|\cdot\|_{\infty, {\mathbf{x}}})\bigr)
$$
and  $\|g\|_{\infty, {\mathbf{x}}} := \max\limits_{1\le j\le m}|g(x_j)|$.
Further, to estimate $\gamma_{2, 1}(G, \|\cdot\|_{\infty, {\mathbf{x}}})$, we use  Dudley's entropy bound \eqref{Dudley} to get
$$
\gamma_{2, 1}(G, \|\cdot\|_{\infty, {\mathbf{x}}})
\le C\sum\limits_{n=0}^\infty 2^{n/2} e_n(G, \|\cdot\|_{\infty, {\mathbf{x}}}).
$$

Let $p_*:=\min\{p, 2\}$. We note that $\Phi\in {\rm (aInc)}_{p_*}$ under the assumptions of Theorems \ref{cor4.1} and \ref{t3.2}.
For $f, g\in X_N^\Phi$, by Lemma \ref{lem0}, we have
\begin{eqnarray*}
&&\bigl\|(\Phi(|f|))^{1/2} - (\Phi(|g|))^{1/2}\bigr\|_{\infty, {\mathbf{x}}}
=
\max\limits_{1\le j\le m}\bigl| \bigl((\Phi(|f(x_j)|))^{1/p_*}\bigr)^{p_*/2} - \bigl((\Phi(|g(x_j)|))^{1/p_*}\bigr)^{p_*/2}\bigr|
\\
&\le&
\max\limits_{1\le j\le m}\bigl|(\Phi(|f(x_j)|))^{1/p_*} - (\Phi(|g(x_j)|))^{1/p_*}\bigr|^{p_*/2}
\\
&\le& (qp_*^{-1})^{p_*/2} \max\limits_{1\le j\le m}\Bigl(|f(x_j) - g(x_j)|\cdot\Bigl|\frac{(\Phi(|f(x_j)|))^{1/p_*}}{|f(x_j)|} + \frac{(\Phi(|g(x_j)|))^{1/p_*}}{|g(x_j)|}\Bigr|\Bigr)^{p_*/2}
\\
&\le& 2q\sup\limits_{h\in X_N^\Phi}\max_{1\le j\le m}\Bigl(\frac{\Phi(|h(x_j)|)}{|h(x_j)|^{p_*}}\Bigr)^{1/2}
\|f-g\|_{\infty, {\mathbf{x}}}^{p_*/2}.
\end{eqnarray*}
Thus,
$$
e_n(G, \|\cdot\|_{\infty, {\mathbf{x}}})\le
2q\sup\limits_{h\in X_N^\Phi}\max_{1\le j\le m}\Bigl(\frac{\Phi(|h(x_j)|)}{|h(x_j)|^{p_*}}\Bigr)^{1/2}
\bigl(e_n(X_N^\Phi, \|\cdot\|_{\infty, {\mathbf{x}}})\bigr)^{p_*/2}.
$$
In light of Lemma \ref{lem0.5},
we have $X_N^\Phi \subset R_\Phi(p_*)X_N^{p_*} =\{f\in X_N\colon \|f\|_{p_*}< R_\Phi(p_*)\}$. Hence,
from the  assumption
$X_N\in NI_{2, \infty}(H)$,
we derive
$$
\|h\|_\infty = \Bigl(\frac{\|h\|_\infty^2}{\|h\|_\infty^{2-p_*}}\Bigr)^{1/p_*}
\le \Bigl(\frac{H^2\|h\|_2^2}{\|h\|_\infty^{2-p_*}}\Bigr)^{1/p_*}
\le H^{2/p_*}\|h\|_{p_*}\le
R_\Phi(p_*) H^{2/p_*}\|h\|_{\Phi}\quad \forall h\in X_N^\Phi.
$$
Since the function $t\mapsto \Phi(t)t^{-p_*}$
is almost increasing
and  $|h(x_j)|\le R_\Phi(p_*) H^{2/p*}$ for $h\in X_N^\Phi$,
we have
$$
\sup_{h\in X_N^\Phi}\max_{1\le j\le m}\frac{\Phi(|h(x_j)|)}{|h(x_j)|^{p_*}}
\le a_\Phi(p_*)\frac{\Phi\bigl(R_\Phi(p_*) H^{2/p_*}\bigr)}{(R_\Phi(p_*))^{p_*} H^2}.
$$
From the inclusion $X_N^\Phi \subset R_\Phi(p_*)X_N^{p_*}$ we get
$$
e_n(X_N^\Phi,\|\cdot\|_{\infty, {\bf x}})
\le R_\Phi(p_*)\cdot e_n(X_N^{p_*},\|\cdot\|_{\infty, {\bf x}}).
$$
Therefore,
\begin{eqnarray*}
\sum\limits_{n=0}^\infty 2^{n/2} e_n(G, \|\cdot\|_{\infty, {\mathbf{x}}})
&\le&
2q(a_\Phi(p_*))^{1/2}\Bigl(\frac{\Phi\bigl(R_\Phi(p_*) H^{2/p_*}\bigr)}{(R_\Phi(p_*))^{p_*} H^2}\Bigr)^{1/2}
\sum\limits_{n\ge 0} 2^{n/2}\bigl(e_n(X_N^\Phi, \|\cdot\|_{\infty, {\bf x}})\bigr)^{p_*/2}
\\
&\le&
2q(a_\Phi(p_*))^{1/2}\Bigl(\frac{\Phi\bigl(R_\Phi(p_*) H^{2/p_*}\bigr)}{H^2}\Bigr)^{1/2}
\sum\limits_{n\ge 0} 2^{n/2}\bigl(e_n(X_N^{p_*}, \|\cdot\|_{\infty, {\bf x}})\bigr)^{p_*/2}.
\end{eqnarray*}
By Lemma \ref{lem4}, we have
$$
\sum\limits_{n\ge 0} 2^{n/2}\bigl(e_n(X_N^{p_*}, \|\cdot\|_{\infty, {\bf x}})\bigr)^{p_*/2}
\le C_1\sum\limits_{n\le [\log N]} 2^{n/2}\bigl(e_n(X_N^{p_*}, \|\cdot\|_{\infty, {\bf x}})\bigr)^{p_*/2}
$$
with  some numerical constant $C_1>0$.
Thus, we arrive at  the following statement.

\begin{proposition}\label{cor-key}
Let $p\in [1, +\infty)$, $q\in[p, +\infty)$,
$p_*:=\min\{p, 2\}$,
$N\ge 1$, and $H\ge 1$.
There is a numerical constant $c>0$ such that,
for every ${\bf \Phi}$-function $\Phi\in {\rm (aInc)}_p\cap {\rm (Dec)}_q$
and for every $N$-dimensional subspace
$X_N\in NI_{2, \infty}(H)$, one has
\begin{equation}\label{eq-key}
\mathbb{E}\Bigl(\sup\limits_{f\in X_N^\Phi}
\Bigl|\frac{1}{m}\sum\limits_{j=1}^m\Phi(|f(x_j)|) - \int_\Omega\Phi(|f|)\, d\mu\Bigr|
\Bigr)\le c(A+A^{1/2}),
\end{equation}
where
$$
A=\frac{1}{m}q^2a_\Phi(p_*)
\frac{\Phi\bigl(R_\Phi(p_*) H^{2/p_*}\bigr)}{H^2}
\mathbb{E}\Bigl(\Bigl(\sum\limits_{n\le [\log N]} 2^{n/2}\bigl(e_n(X_N^{p_*}, \|\cdot\|_{\infty, {\bf x}})\bigr)^{p_*/2}\Bigr)^2\Bigr),
$$
$\|f\|_{\infty, {\bf x}}:=\max\limits_{1\le j\le m}|f(x_j)|$, and  points $x_1, \ldots, x_m$ are choosing
randomly, independently, and distributed according to the measure $\mu$.
\end{proposition}

To complete the proofs of Theorems \ref{cor4.1} and \ref{t3.2}, we will consider three separate cases in the following subsections:
$p\ge 2$, $p\in(1, 2)$, and $p=1$.

\subsection{Proof of Theorem \ref{cor4.1}: the case  $p\ge 2$}\label{ssect3-3}

We note that in this case $p_*=2$.

\begin{lemma}\label{cor2.1}
Let $p\ge 2$, $q\in[p, +\infty)$, $N\ge 1$, $H\ge 1$.
There is a numerical constant $c>0$ such that,
for every ${\bf \Phi}$-function $\Phi\in {\rm(aInc)}_p\cap{\rm(Dec)}_q$
and for every $N$-dimensional subspace $X_N\in NI_{2, \infty}(H)$, one has
$$
\mathbb{E}\Bigl(\sup\limits_{f\in X_N^\Phi}
\Bigl|\frac{1}{m}\sum\limits_{j=1}^m\Phi(|f(x_j)|) - \int_\Omega\Phi(|f|)\, d\mu\Bigr|
\Bigr)\le c(A+A^{1/2}),
$$
where
\begin{equation}\label{eq-A1}
A=\frac{q^2a_\Phi(2)\Phi\bigl(R_\Phi(2)H\bigr)\log m(\log 2N)^2}{m}.
\end{equation}
\end{lemma}

\begin{proof}
By Proposition \ref{cor-key}, it is enough to estimate
$$
\sum\limits_{n\le [\log N]} 2^{n/2}e_n(X_N^2, \|\cdot\|_{\infty, {\bf x}})
$$
for a fixed set of points ${\bf x}=\{x_1, \ldots, x_m\}\subset \Omega$.
Let $\|h\|_{\infty, {\bf x}}:=\max\limits_{1\le j\le m}|h(x_j)|$.
Taking into account  the dual Sudakov bound for the
entropy numbers of the Euclidean ball (see \cite[L. 8.3.6]{Tal}),  we have
$$
e_n(X_N^2,\|\cdot\|_{\infty, {\bf x}})
\le
c\, 2^{-n/2}\mathbb{E}_g\bigl\|\sum_{k=1}^{N}g_ku_k\bigr\|_{\infty, X},
$$
where $c$ is a numerical constant,
$g=(g_1,\ldots, g_N)$ is a standard Gaussian random vector,
and $\{u_1,\ldots, u_N\}$ is any orthonormal basis in $X_N$.
Using now the
estimate for the expectation of the maximum of Gaussian random variables
(see \cite[Prop. 2.4.16]{Tal}) and the Nikolskii inequality,
we derive that
\begin{eqnarray*}
\mathbb{E}_g\bigl\|\sum_{k=1}^{N}g_ku_k\bigr\|_{\infty, {\bf x}}
&=&
\mathbb{E}_g\max\limits_{1\le j\le m}\bigl|\sum_{k=1}^{N}g_ku_k(x_j)\bigr|
\\
&\le& c_1 \max\limits_{1\le j\le m}\bigl(\sum\limits_{k=1}^N|u_k(x_j)|^2\bigr)^{1/2}(\log m)^{1/2}
\le  c_1 H(\log m)^{1/2}.
\end{eqnarray*}
Thus,
\begin{equation}\label{entropy}
e_n(X_N^2,\|\cdot\|_{\infty, {\bf x}})
\le
c_2 H2^{-n/2}(\log m)^{1/2}.
\end{equation}
and we have
$$
\sum\limits_{n\le [\log N]} 2^{n/2}e_n(X_N^2, \|\cdot\|_{\infty, {\bf x}})
\le
c_2 H(\log m)^{1/2} \log 2N.
$$
The lemma is proved.
\end{proof}

{\sc Proof of Theorem \ref{cor4.1} for $p\ge 2$.}
By Lemma \ref{lem5}, Lemma \ref{cor2.1},
and Chebyshev's inequality,
it is sufficient to choose $m$ such that
$$c(A+A^{1/2})\le \varepsilon/2,$$ where $A$ is given by  \eqref{eq-A1}.
Since $\varepsilon\in (0, 1/2]$, it is enough to see that $$A \varepsilon^{-2}\le (4c)^{-2}.$$
Let now
$$
m\ge C_\Phi\varepsilon^{-2}(\log\varepsilon^{-1}) \Phi(H)(\log 2N)^2 \log 2H^2
$$
with sufficiently large constant $C_\Phi\ge 2$, which will be specified later.

We note that $\Phi\in {\rm (Dec)}_q$ with some
$q:=q(\Phi)\in[1, +\infty)$. Then,
since 
 $R_\Phi(2)\ge1$ (see \eqref{eq-R}),
we have
$$
\Phi\bigl(R_\Phi(2)H\bigr)\le (R_\Phi(2))^q\Phi(H).
$$
Thus, using that $t\mapsto \frac{\log t}{t}$ is decreasing for $t\ge 3$, we have
$$
A \varepsilon^{-2}=\varepsilon^{-2}\frac{\log m}{m}
q^2a_\Phi(2)\Phi\bigl(R_\Phi(2)H\bigr)(\log 2N)^2
\le
\frac{q^2a_\Phi(2)(R_\Phi(2))^q}{C_\Phi} J,
$$
where
$$J:=\frac{\log C_\Phi + 2\log\varepsilon^{-1}+\log\log\varepsilon^{-1}+
\log\Phi(H) + 2\log\log 2N+\log\log 2H^2}
{(\log\varepsilon^{-1})\log 2H^2}.
$$
Noting that $H^2\ge N$ and
$\log\log t\le \log t$ for $t>2$,  we estimate
$$
J\le
\log C_\Phi+\frac{\log\Phi(H)} {\log 2H^2}+C
$$
with some numerical constant $C>0$.

Finally, in view of the inequality $\Phi(H)\le \Phi(1)H^q$ (note that $H\ge 1$),
we have
$$
\log\Phi(H)\le
\log(\Phi(1))+q/2\log H^2.
$$
Therefore,
$$
A \varepsilon^{-2}\le
\frac{q^2a_\Phi(2)(R_\Phi(2))^q}{C_\Phi}\Bigl(\log C_\Phi+ \log(\Phi(1)) + q/2 + C \Bigr).
$$
By choosing the constant $C_\Phi$ sufficiently large, we can make the right-hand side of the last inequality smaller than $(4c)^{-2}$. This implies that $A \varepsilon^{-2}\le (4c)^{-2}$, completing  the proof.

\qed

\subsection{Proof of Theorem \ref{cor4.1}: the case $p\in (1, 2)$}
\label{ssect3-4}

In this case we need the following lemma.

\begin{lemma}\label{cor3.2}
Let $p\in (1, 2)$, $q\in[p, +\infty)$, $N\ge 1$, $H\ge 1$.
There is a constant $c(p)>0$, depending only on $p$, such that,
for every ${\bf \Phi}$-function $\Phi\in {\rm (aInc)}_p\cap {\rm (Dec)}_q$
and for every $N$-dimensional subspace $X_N\in NI_{2, \infty}(H)$, one has
$$
\mathbb{E}\Bigl(\sup\limits_{f\in X_N^\Phi}
\Bigl|\frac{1}{m}\sum\limits_{j=1}^m\Phi(|f(x_j)|) - \int_\Omega\Phi(|f|)\, d\mu\Bigr|
\Bigr)\le c(p)(A+A^{1/2}),
$$
where
\begin{equation}\label{eq-A2}
A=\frac{q^2a_\Phi(p)
\Phi\bigl(R_\Phi(p) H^{2/p}\bigr)
(\log 2H^2)^{1-p/2}(\log m)^{p/2}(\log 2N)^2}{m}.
\end{equation}
\end{lemma}

\begin{proof}
In light of Proposition \ref{cor-key},
our goal is to estimate
$$
\sum\limits_{n\le [\log N]} 2^{n/2}\bigl(e_n(X_N^p, \|\cdot\|_{\infty, {\bf x}})\bigr)^{p/2}.
$$
It is known (see, e.g., \cite[L. 4.10]{Kos}) that
$$
e_n(X_N^p,\|\cdot\|_{\infty, {\bf x}})
\le C(p)  H^{2/p}2^{-n/p}
(\log 2H^2)^{\frac{1}{p}-\frac{1}{2}}(\log m)^{1/2}
$$
with a constant $C(p)>0$ depending only on $p$.
This yields
\begin{eqnarray*}
\sum\limits_{n\le [\log N]} 2^{n/2}\bigl(e_n(X_N^\Phi, \|\cdot\|_{\infty, {\bf x}})\bigr)^{p/2}
&\le&
C_1(p)  H
(\log 2H^2)^{\frac{1}{2}-\frac{p}{4}}(\log m)^{\frac{p}{4}}\log 2N.
\end{eqnarray*}
This completes  the proof.
\end{proof}

{\sc Proof of Theorem \ref{cor4.1} for $p\in (1, 2)$.}
The argument is similar to the one in the proof in the case $p\ge2$.
Taking into account  Lemma \ref{cor3.2},
 it suffices to take $m$ such that
$c(p)(A+A^{1/2})\le \varepsilon/2$
  and
  \begin{equation}\label{eqvsp}
      \varepsilon^{-2} A\le (4c(p))^{-2}.
  \end{equation}
Let now
$$
m\ge C_{\Phi, p}
\varepsilon^{-2}(\log\varepsilon^{-1})^{p/2} \Phi\bigl(H^{2/p}\bigr)(\log 2N)^2\log 2H^2
$$
with sufficiently large constant $C_{\Phi, p}\ge 2$, which will be specified later.
Since the function $t\mapsto \Phi(t)t^{-q}$ is decreasing on $(0, +\infty)$
with some $q:=q(\Phi)\in[1, +\infty)$ and $R_\Phi(p)\ge1$,
we have
$$
\Phi\bigl(R_\Phi(p)H^{2/p}\bigr)\le (R_\Phi(p))^q\Phi\bigl(H^{2/p}\bigr),
$$
which implies
\begin{eqnarray*}
\varepsilon^{-2}A&=&\varepsilon^{-2}\frac{(\log m)^{p/2}}{m}
q^2a_\Phi(p)
\Phi\bigl(R_\Phi(p) H^{2/p}\bigr)
(\log H^2)^{1-p/2}(\log 2N)^2
\\&\le&
\frac{q^2a_\Phi(p)(R_\Phi(p))^q}{C_{\Phi, p}}J,
\end{eqnarray*}
where
$$J=
\frac{\bigl(\log C_{\Phi, p} + 2\log\varepsilon^{-1}+\frac{p}{2}\log\log\varepsilon^{-1}+
\log\Phi\bigl(H^{2/p}\bigr) + 2\log\log 2N+\log\log 2H^2\bigr)^{p/2}}
{(\log\varepsilon^{-1})^{p/2}(\log 2 H^2)^{p/2}}.
$$
It is easy to see that, for some  positive $C$,

$$J^{2/p}\le
\log C_{\Phi, p}
+\frac{\log\Phi\bigl(H^{2/p}\bigr)}{\log 2 H^2}+C.
$$
Further, taking into account that
$\Phi\bigl(H^{2/p}\bigr)\le \Phi(1)H^{2q/p}$, we  have
$$
\log\Phi\bigl(H^{2/p}\bigr)\le
\log(\Phi(1)) + q/p\log H^2.
$$
Therefore,
$$
\varepsilon^{-2}A
\le
\frac{q^2a_\Phi(p)(R_\Phi(p))^q}{C_{\Phi, p}}\bigl(\log C_{\Phi, p} + \log(\Phi(1)) +
q/p + C\bigr)^{p/2}.
$$
Taking  $C_{\Phi, p}$ large  enough, we arrive at
\eqref{eqvsp}.
\qed

\subsection{Proof of Theorem \ref{t3.2}}\label{ssect3-5}

Let $q:=q(\Phi)\in[1, +\infty)$ be such that $\Phi\in {\rm (Dec)}_q$.
Since
$$
\|f\|_2^2\le \|f\|_\infty\|f\|_1\le H\|f\|_2\|f\|_1\quad \forall f\in X_N
$$
we have $X_N^1\subset HX_N^2$.
Thus,
$$
e_n(X_N^1,\|\cdot\|_{\infty, {\bf x}})
\le
H e_n(X_N^2,\|\cdot\|_{\infty, {\bf x}})
\le
c H^2 2^{-n/2}(\log m_0)^{1/2}.
$$
for every fixed set of points ${\bf x}:=\{x_1, \ldots, x_{m_0}\}$,
where we have used estimate \eqref{entropy}.
Using Proposition \ref{cor-key}, we derive
$$
\mathbb{E}\Bigl(\sup\limits_{f\in X_N^\Phi}
\Bigl|\frac{1}{m_0}\sum\limits_{j=1}^{m_0}\Phi(|f(x_j)|) - \int_\Omega\Phi(|f|)\, d\mu\Bigr|
\Bigr)\le c_1(A+A^{1/2}),
$$
where
$$
A=\frac{1}{m_0}q^2a_\Phi(1)
\Phi\bigl(R_\Phi(1) H^2\bigr)N^{1/2}(\log m_0)^{1/2}.
$$
Using
the fact that the function $t\mapsto \Phi(t)t^{-q}$ is decreasing
and
$R_\Phi(1)H^2\ge1$,
we note that
\begin{equation}\label{eq-dec}
\Phi\bigl(R_\Phi(1)H^2\bigr)\le (R_\Phi(1))^qH^{2q}\Phi(1).
\end{equation}
Therefore, since $
\log m_0\le m_0$,
$$
\mathbb{E}\Bigl(\sup\limits_{f\in X_N^\Phi}
\Bigl|\frac{1}{m_0}\sum\limits_{j=1}^{m_0}\Phi(|f(x_j)|) - \int_\Omega\Phi(|f|)\, d\mu\Bigr|
\Bigr)\le c_2(\Phi)(A_1+A_1^{1/2}),
$$
where
$$
A_1=\frac{H^{2q}N^{1/2}}{\sqrt{m_0}}.
$$
On the other hand,  Corollary \ref{cor2.1} implies
that
$$
\mathbb{E}\Bigl(\sup\limits_{f\in X_N^2}
\Bigl|\frac{1}{m_0}\sum\limits_{j=1}^{m_0}|f(x_j)|^2 - \int_\Omega|f|^2\, d\mu\Bigr|
\Bigr)\le c_3(A_2+A_2^{1/2}),
$$
where
$$
A_2=\frac{H\log m_0(\log 2N)^2}{m_0}\le 16\frac{HN}{\sqrt{m_0}}.
$$
Thus, for any $\alpha\in(0, 1/2]$,
there exist a number $m_0\le c_4(\Phi)\alpha^{-4} H^{4q}N^2$ and a set of points
${\bf x}_0:=\{x_1^0, \ldots, x_{m_0}^0\}\subset \Omega$
such that
\begin{equation}\label{eq3.2}
a_\Phi(1)^{-1}(1-\alpha)\|f\|_{L^\Phi(\mu)}\le \|f\|_{L^\Phi(\nu)}\le
a_\Phi(1)(1+\alpha)\|f\|_{L^\Phi(\mu)}
\quad \forall f\in X_N
\end{equation}
and
$$
\frac{1}{2}\|f\|_{L^2(\mu)}\le \|f\|_{L^2(\nu)}\le\frac{3}{2}\|f\|_{L^2(\mu)}
\quad \forall f\in X_N,
$$
where $\nu = \frac{1}{m_0}\sum\limits_{j=1}^{m_0}\delta_{x_j^0}$.
Then
$$
\|f\|_{L^\infty(\nu)}=\sup\limits_{x\in \{x_1, \ldots, x_{m_0}\}}|f(x)|
\le \|f\|_\infty\le H\|f\|_2\le 2H\|f\|_{L^2(\nu)}.
$$
Let $X_N^\Phi(\nu)$, $X_N^2(\nu)$, and $X_N^1(\nu)$ be the (open) unit balls in $X_N$
with respect to the norms of spaces $L^\Phi(\nu)$, $L^2(\nu)$, and $L^1(\nu)$, respectively.
In light of  estimate \eqref{entropy},
$$
e_n(X_N^2(\nu),\|\cdot\|_{L^\infty(\nu)})\le
c H 2^{-n/2}(\log m_0)^{1/2}.
$$
We now note that Lemma 3.3 from \cite{DPSTT2} implies that
$$
e_n(X_N^1(\nu), \|\cdot\|_{L^\infty(\nu)})\le C R 2^{-n}\quad \forall n\in \mathbb{N}
$$
if
$$
e_n(X_N^2(\nu), \|\cdot\|_{L^\infty(\nu)})\le (R 2^{-n})^{1/2}\quad \forall n\in \mathbb{N}.
$$
Hence,
$$
e_n(X_N^1(\nu), \|\cdot\|_{L^\infty(\nu)})\le c_5(\Phi) H^2(\log \alpha^{-1}) (\log 2H^2) 2^{-n}\quad \forall n\in \mathbb{N},
$$
where we have used the estimate $H\ge \sqrt{N}$.
We now apply Proposition \ref{cor-key} with the set ${\bf x}_0$ and with the measure $\nu$.
We note that for any ${\bf x}\subset {\bf x}_0$, one has
\begin{eqnarray*}
\sum\limits_{n\le [\log N]} 2^{n/2}\bigl(e_n(X_N^1(\nu), \|\cdot\|_{\infty, {\bf x}})\bigr)^{1/2}
&\le&
\sum\limits_{n\le [\log N]} 2^{n/2}\bigl(e_n(X_N^1(\nu), \|\cdot\|_{\infty, {\bf x}_0})\bigr)^{1/2}
\\
&\le&
c_6(\Phi)H(\log \alpha^{-1})^{1/2} (\log 2H^2)^{1/2} (\log 2N).
\end{eqnarray*}
Thus, we obtain
$$
\mathbb{E}\Bigl(\sup\limits_{f\in X_N^\Phi(\nu)}
\Bigl|\frac{1}{m}\sum\limits_{j=1}^{m}\Phi(|f(x_j)|) - \int_{{\bf x}_0}\Phi(|f|)\, d\nu\Bigr|
\Bigr)\le c_7(\Phi)(A+A^{1/2}),
$$
where
$$
A=\frac{(\log \alpha^{-1})\Phi\bigl(H^2\bigr)(\log 2H^2) (\log 2N)^2}{m}.
$$
Therefore, for any $\beta\in(0, 1/2]$, there exists
$m\le c_8(\Phi)\beta^{-2} (\log \alpha^{-1})\Phi\bigl(H^2\bigr)(\log 2H^2) (\log 2N)^2$
such that $c_7(\Phi)(A+A^{1/2})\le \beta/2$. For this $m$,
there is a set of points ${\bf x}:=\{x_1, \ldots, x_m\}\subset {\bf x}_0$
such that
$$
a_\Phi(1)^{-1}(1-\beta)\|f\|_{L^\Phi(\nu)}\le \|f\|_{\Phi, {\bf x}}\le
a_\Phi(1)(1+\beta)\|f\|_{L^\Phi(\nu)}
\quad \forall f\in X_N.
$$
Finally, taking $\alpha = \beta= \varepsilon/3$
and combining this bound with \eqref{eq3.2}, we obtain the required statement.
\qed

\vspace{0.2cm}

\section{Descritization under relaxed assumptions on ${\bf \Phi}$-functions}
\label{sect4}

In this section, we show  that the conditions on the ${\bf \Phi}$-function in Theorem \ref{cor4.1} can be relaxed while still achieving an effective discretization result. Specifically,
we demonstrate   that the assumption $\Phi\in {\rm (aInc)}_p\cap{\rm (Dec)}$ can be replaced with the less restrictive condition 
$\Phi\in {\rm (aInc)}_p(\infty)\cap{\rm (aDec)}(\infty)$, yielding the same discretization result up to constant factors.

\begin{theorem}\label{cor4.2}
Let $p\in (1, \infty)$, $N\ge 1$, $H\ge 1$.
For every ${\bf \Phi}$-function $\Phi$, satisfying the condition
$\Phi\in {\rm (aInc)}_p(\infty)\cap{\rm (aDec)}(\infty)$,
there exist positive constants $c:=c(\Phi, p)$, $C:=C(\Phi, p)$,
$C_1:=C_1(\Phi, p)$, and $C_2:=C_2(\Phi, p)$,
depending only on $\Phi$ and $p$, such that
for every $N$-dimensional subspace $X_N\in NI_{2, \infty}(H)$ and for every
$$m\ge C
\Phi\bigl(cH^{\frac{2}{\min\{p, 2\}}}\bigr)(\log 2N)^2\log 2H^2,$$
there is a subset ${\bf x}:=\{x_1, \ldots, x_m\}\subset \Omega$
of cardinality $m$
such that
$$
C_1\|f\|_\Phi\le \|f\|_{\Phi, {\bf x}}\le C_2\|f\|_\Phi
\quad \forall f\in X_N.
$$
\end{theorem}

\begin{proof}
Since $\Phi\in {\rm (aInc)}_p(\infty)\cap{\rm (aDec)}(\infty)$,
there are  $t_*:=t_*(\Phi)\in (0, +\infty)$
and  $q:=q(\Phi)> p$
such that the mapping
$t\mapsto \Phi(t)t^{-p}$ is almost increasing with a constant $a:=a_\Phi(p)$
on $(t_*, \infty)$ and $t\mapsto \Phi(t)t^{-q}$ is almost decreasing with a constant $b:=b_\Phi(q)$
on $(t_*, \infty)$.
Let $t_0=2t_*$.
We consider a new ${\bf \Phi}$-function $\Phi_1$ such that
$$
\Phi_1(t): =
\left\{
\begin{aligned}
  \frac{\Phi(t_0)}{t_0^p}t^p, &\ t\in(0, t_0] \\
  \Phi(t), &\ t\ge t_0 \\
\end{aligned}
\right..
$$
Now we claim  that  $\Phi_1\in {\rm (aInc)}_p\cap{\rm (aDec)}$.
Indeed, let $0<s<t$ and  $\varphi(\tau):=\Phi_1(\tau)\tau^{-p}$.
If $t\le t_0$, then $\varphi(s)=\varphi(t)\le a\varphi(t)$.
If $s\ge t_0$, then $\varphi(s) = \Phi(s)s^{-p}\le a\Phi(t)t^{-p}=a\varphi(t)$.
If $s< t_0 <t$, then $\varphi(s)= \varphi(t_0)\le a\varphi(t)$.
Let now $\psi(\tau):=\Phi_1(\tau)\tau^{-q}$.
If $t\le t_0$, then $\psi(t)=ct^{p-q}\le cs^{p-q}=\psi(s)\le b\psi(s)$ where $c:=\frac{\Phi(t_0)}{t_0^p}$.
If $s\ge t_0$, then $\psi(t) = \Phi(t)t^{-q}\le b\Phi(s)s^{-q}=b\psi(s)$.
If $s< t_0 <t$, then $\psi(t) \le b\psi(t_0)\le b\psi(s)$, completing the proof of the claim.

Using  Lemma~\ref{lem-equiv},
we can find
a convex ${\bf \Phi}$-function
$\Phi_2\in {\bf \Phi}_c$, which is equivalent to $\Phi_1$.
By \cite[L.~2.1.9]{HH19}, we have $\Phi_2\in {\rm (aInc)}_p\cap{\rm (aDec)}$.
Moreover, for a convex ${\bf \Phi}$-function, the conditions $\Phi_2\in {\rm (aDec)}$
and $\Phi_2\in {\rm (Dec)}$ are equivalent (see \cite[L.~2.2.6]{HH19}).
Thus, $\Phi_2\in {\rm (aInc)}_p\cap{\rm (Dec)}$, and
by Theorem \ref{cor4.1}, for every
$$
m\ge C_{\Phi_2, p}
\Phi_2\bigl(H^{\frac{2}{\min\{p, 2\}}}\bigr)(\log 2N)^2\log 2H^2,
$$
there is a subset ${\bf x}:=\{x_1, \ldots, x_m\}\subset \Omega$
of cardinality $m$
such that
$$
\frac{1}{2}\|f\|_{\Phi_2}\le \|f\|_{\Phi_2, {\bf x}}\le \frac{3}{2}\|f\|_{\Phi_2}
\quad \forall f\in X_N.
$$
Let $c=c(\Phi, p): = \max\{L, t_0\}$, where $L$ is the constant from the
equivalence
$\Phi_1(L^{-1}x)\le\Phi_2(x)\le \Phi_1(L x)$.
Then
$$
\Phi_2\bigl(H^{\frac{2}{\min\{p, 2\}}}\bigr)
\le \Phi_1\bigl(LH^{\frac{2}{\min\{p, 2\}}}\bigr)
\le \Phi_1\bigl(cH^{\frac{2}{\min\{p, 2\}}}\bigr)
= \Phi\bigl(cH^{\frac{2}{\min\{p, 2\}}}\bigr).
$$
Finally, we note that for any probability measure $\nu$,
one has
\begin{eqnarray*}
&&\int_\Omega\Phi_2\Bigl(\frac{|f|}{2L(\Phi(t_0) + 1) \|f\|_{L^\Phi(\nu)}}\Bigr)\, d\nu
\le
\frac{1}{\Phi(t_0) + 1}
\int_\Omega\Phi_2\Bigl(\frac{|f|}{2L\|f\|_{L^\Phi(\nu)}}\Bigr)\, d\nu
\\&&\qquad\qquad\qquad\qquad\qquad\qquad\qquad\quad
\le
\frac{1}{\Phi(t_0) + 1}
\int_\Omega\Phi_1\Bigl(\frac{|f|}{2\|f\|_{L^\Phi(\nu)}}\Bigr)\, d\nu
\\
&\le&
\frac{1}{\Phi(t_0) + 1}
\Bigl(\Phi_1(t_0)\nu\bigl(|f|(2\|f\|_{L^\Phi(\nu)})^{-1}\le t_0\bigr)
+ \int_\Omega\Phi\Bigl(\frac{|f|}{2\|f\|_{L^\Phi(\nu)}}\Bigr)\, d\nu\Bigr)\le1,
\end{eqnarray*}
where in the first inequality we have used the convexity of $\Phi_2$,
in the second one the equivalence between $\Phi_1$ and $\Phi_2$,
and in the last one the first implication in \eqref{eq-ball}.
Therefore, we arrive at  $\|f\|_{L^{\Phi_2}(\nu)}\le 2L(\Phi(t_0) + 1) \|f\|_{L^\Phi(\nu)}$.

Similarly,
\begin{eqnarray*}
&&\int_\Omega\Phi\Bigl(\frac{|f|}{2La_\Phi(1)(\Phi(t_0) + 1) \|f\|_{L^{\Phi_2}(\nu)}}\Bigr)\, d\nu
\le
\frac{1}{\Phi(t_0) + 1}
\int_\Omega\Phi\Bigl(\frac{|f|}{2L\|f\|_{L^{\Phi_2}(\nu)}}\Bigr)\, d\nu
\\
&\le&
\frac{1}{\Phi(t_0) + 1}
\Bigl(\Phi(t_0)\nu\bigl(|f|(2L\|f\|_{L^{\Phi_2}(\nu)})^{-1}\le t_0\bigr)
+ \int_\Omega\Phi_1\Bigl(\frac{|f|}{2L\|f\|_{L^{\Phi_2}(\nu)}}\Bigr)\, d\nu\Bigr)
\\
&\le&
\frac{1}{\Phi(t_0) + 1}
\Bigl(\Phi(t_0)
+ \int_\Omega\Phi_2\Bigl(\frac{|f|}{2\|f\|_{L^{\Phi_2}(\nu)}}\Bigr)\, d\nu\Bigr)
\le1,
\end{eqnarray*}
where in the first inequality we have used \eqref{eq-est}.
Thus,
$\|f\|_{L^\Phi(\nu)}\le 2a_\Phi(1)L(\Phi(t_0) + 1) \|f\|_{L^{\Phi_2}(\nu)}$.
Therefore, for every
$$
m\ge C(\Phi, p)
\Phi\bigl(cH^{\frac{2}{\min\{p, 2\}}}\bigr)(\log 2N)^2\log 2H^2 \;\;\;\hbox{ with } \;\;\;
C(\Phi, p) = C_{\Phi_2, p},
$$
there is a subset ${\bf x}:=\{x_1, \ldots, x_m\}\subset \Omega$
of cardinality $m$
such that, for every $f\in X_N$, one has
\begin{eqnarray*}
\frac{1}{8a_\Phi(1)L^2(\Phi(t_0) + 1)^2}\|f\|_\Phi&\le&
\frac{1}{4L(\Phi(t_0) + 1)}\|f\|_{\Phi_2}\le \frac{1}{2L(\Phi(t_0) + 1)}\|f\|_{\Phi_2, {\bf x}}
\\ &\le&
\|f\|_{\Phi, {\bf x}}\le
2a_\Phi(1)L(\Phi(t_0) + 1)\|f\|_{\Phi_2, {\bf x}}\\ &\le&
3a_\Phi(1)L(\Phi(t_0) + 1)\|f\|_{\Phi_2}
\le 6a_\Phi(1)L^2(\Phi(t_0) + 1)^2\|f\|_{\Phi}.
\end{eqnarray*}
The proof  is now complete.
\end{proof}

\vspace{0.2cm}

\section{One-sided weighted discretization for an arbitrary subspace} \label{sect6}

The aim of this section is to obtain one-sided discretization inequalities without additional conditions on the subspace $X_N$,
that is,  we do not assume that  Nikolskii's inequality \eqref{eq-Nik} holds.

\subsection{Bounds of integral Orlicz norms by discrete Orlicz norms}
\label{ssect6-1}
We start with the following analogue of Lewis's change of density lemma
(see \cite{Lew78} and \cite{SZ01}).

\begin{lemma}\label{lem5.1}
Let ${\bf x}=\{x_1, \ldots, x_m\}$,  $\nu:=\sum\limits_{j=1}^{m}\lambda_j\delta_{x_j}$ be a discrete positive measure on ${\bf x}$,
and  $X_N$ be an $N$-dimensional subspace of functions on ${\bf x}$.
Assume that $\Phi$ is an ${\bf \Phi}$-prefunction with continuous derivative $\varphi$.
 There exist a constant $c>0$ and a basis $v_1, \ldots, v_N$
in $X_N$ such that
$$
\int_{{\bf x}'} \varphi\Bigl(
\Bigl(\sum\limits_{k=1}^{N}|v_k(x)|^2\Bigr)^{1/2}\Bigr)
\Bigl(\sum\limits_{k=1}^{N}|v_k(x)|^2\Bigr)^{-1/2}
v_{r}(x)v_{r'}(x)\, \nu(dx) = c\delta_{r, r'}
$$
and
$$
\int_{{\bf x}} \Phi\Bigl(
\Bigl(\sum\limits_{k=1}^{N}|v_k(x)|^2\Bigr)^{1/2}\Bigr)\, \nu(dx)=1,
$$
where ${\bf x}':=\bigl\{x\in {\bf x}\colon \sum\limits_{k=1}^{N}|v_k(x)|^2\ne 0\bigr\}$.
\end{lemma}

\begin{proof}
The proof repeats the argument
by Schechtman and  Zvavitch in \cite[Th. 2.1]{SZ01}.
Without loss of generality, we assume that
for each point $x_j$ there is a function $f\in X_N$
such that $f(x_j)\ne 0$.
Let $\{u_1, \ldots, u_N\}$ be a basis in $X_N$.
For an $N\times N$ matrix $B:=(b_{k, l})$, we define
$$
G(B):= \int_{{\bf x}} \Phi\Bigl(
\Bigl(\sum\limits_{k=1}^{N}\Bigl|\sum_{l=1}^{N}
b_{k, l}u_l(x)\Bigr|^2\Bigr)^{1/2}\Bigr)\, \nu(dx)=
\sum_{j=1}^{m}\lambda_j\Phi\Bigl(
\Bigl(\sum\limits_{k=1}^{N}\Bigl|\sum_{l=1}^{N}
b_{k, l}u_l(x_j)\Bigr|^2\Bigr)^{1/2}\Bigr).
$$
Clearly,  $G$ is a continuous function.
Since the set $\{B\colon G(B)=1\}$ is compact,
there exists
a matrix $A:=(a_{k, l})$  such that the maximum value of $\det B$
under the condition $G(B)=1$ is attained on $A$.
For $t\ge0$, one has
$$
G(t I) =
\sum_{j=1}^{m}\lambda_j\Phi\Bigl(t
\Bigl(\sum\limits_{k=1}^{N}|u_k(x_j)|^2\Bigr)^{1/2}\Bigr).
$$
Moreover, $t\mapsto G(tI)$ is a continuous function on
$[0, +\infty)$, $G(0) = 0$, and
$\lim\limits_{t\to +\infty}G(tI) = +\infty$.
Thus, there is $t_0\in(0, +\infty)$ such that $G(t_0I)=1$, which implies that  $\det A\ge t_0^N>0$.
Therefore, the functions
$v_k:=\sum_{l=1}^{N} a_{k, l}u_l$ form a basis in $X_N$.
In particular,
$$
\sum\limits_{k=1}^{N}\Bigl|\sum_{l=1}^{N}
a_{k, l}u_l(x_j)\Bigr|^2 = \sum\limits_{k=1}^{N}|v_k(x_j)|^2>0
$$
for every $j\in \{1, \ldots, m\}$ and therefore, the function $G$
is continuously differentiable on some neighborhood
of the point $A$.
We note that, for fixed $k_0$ and $l_0$,
$$
\frac{\partial}{\partial b_{k_0l_0}}G(B)
=
\int_{{\bf x}}\varphi\Bigl(
\Bigl(\sum\limits_{k=1}^{N}\Bigl|\sum_{l=1}^{N}
b_{k, l}u_l\Bigr|^2\Bigr)^{1/2}\Bigr)
\Bigl(\sum\limits_{k=1}^{N}\Bigl|\sum_{l=1}^{N}
b_{k, l}u_l\Bigr|^2\Bigr)^{-1/2}
\Bigl(\sum_{s=1}^{N}
b_{k_0, s}u_s\Bigr)u_{l_0}\, d\nu.
$$
Using the method of Lagrange multipliers,
the matrix $C:=
\bigl(\frac{\partial}{\partial b_{k_0l_0}}G(B)\bigr)_{k_0, l_0}\Bigl|_{B=A}$
 coincides, up to a constant, with the matrix
$\bigl(\frac{\partial}{\partial b_{k_0l_0}}\det(B)\Bigl|_{B=A}\bigr)_{k_0, l_0}$.
The latter matrix is equal to $(\det A)\cdot (A^T)^{-1}$.
Thus, $CA^T$ coincides, up to a constant factor, with the unit  matrix. Thus,  
\begin{eqnarray*}
&&\int_{{\bf x}}\varphi\Bigl(
\Bigl(\sum\limits_{k=1}^{N}|v_k(x)|^2\Bigr)^{1/2}\Bigr)
\Bigl(\sum\limits_{k=1}^{N}|v_k(x)|\Bigr)^{-1/2}
v_{r}(x)v_{r'}(x)\, \nu(dx)\\
&=&\int_{{\bf x}}\varphi\Bigl(
\Bigl(\sum\limits_{k=1}^{N}\Bigl|\sum_{l=1}^{N}
a_{k, l}u_l\Bigr|^2\Bigr)^{1/2}\Bigr)
\Bigl(\sum\limits_{k=1}^{N}\Bigl|\sum_{l=1}^{N}
a_{k, l}u_l\Bigr|^2\Bigr)^{-1/2}
\Bigl(\sum_{s=1}^{N}
a_{r, s}u_s\Bigr)\Bigl(\sum_{s=1}^{N}
a_{r', s}u_s\Bigr)\, d\nu = c\delta_{r, r'}
\end{eqnarray*}
with some constant $c>0$.
\end{proof}


\begin{theorem}\label{t5.1}
Let $p\in (1, \infty)$.
Let $\Phi$ be a ${\bf \Phi}$-function such that
$\Phi \in {\rm (Inc)}_p\cap{\rm (Dec)}$ and
$$
\sup\limits_{t>0}\bigr(\Phi(t)\Phi(t^{-1})\bigl)<\infty.
$$
Assume that $\varphi:=\Phi'$ is continuous on $(0, +\infty)$
and   there is  a ${\bf \Phi}$-function $\Psi\in {\rm (aInc)}_p\cap{\rm (Dec)}$
such that
\begin{equation}\label{cond}
\frac{\Phi(ts)}{\Phi(t)}\le K_{\Phi, \Psi}\Psi(s)\quad \forall t, s>0
\end{equation}
with $K_{\Phi, \Psi}\ge 1$.
Then there are positive constants $C:=C(\Phi, \Psi, p)$ and $c:=c(\Phi, \Psi, p)$,
depending only on $\Phi$, $\Psi$, and $p$,
such that
for every $N$-dimensional subspace $X_N\subset C(\Omega)$, $1\in X_N$,
there exist a set ${\bf x}:=\{x_1, \ldots, x_m\}\subset \Omega$
of cardinality
$$
m\le c\Psi\bigl(N^\frac{1}{\min\{p,2\}}\bigr)(\log 2N)^3
$$
and positive weights $\lambda=\{\lambda_1, \ldots, \lambda_m\}$,
$\lambda_1+\ldots+\lambda_m=1$, providing the following one-sided discretization inequality
\begin{equation}\label{result}
\|f\|_\Phi\le C\bigl(M_{\Phi, \Psi}(m)\bigr)^{1/p}
\|f\|_{\Psi, {\bf x}, \lambda} \quad \forall f\in X_N,
\end{equation}
where
\begin{eqnarray*}
M_{\Phi, \Psi}(m)&:=&
\max\Bigl\{1, \max\Bigl\{\frac{\Psi(t)}{\Phi(t)}\colon t\in[\Phi^{-1}(1), \Phi^{-1}(m)]\Bigr\}\Bigr\}
\\&=& \max\Bigl\{1,\max\Bigl\{\frac{\Psi\circ\Phi^{-1}(t)}{t}\colon t\in[1, m]\Bigr\}\Bigr\}
\quad \forall m\in(0, +\infty).
\end{eqnarray*}

\end{theorem}

\begin{proof}
Since $X_N$ is a finite-dimensional space of continuous functions,
there is a constant $H>0$ such that
$X_N\in NI_{2,\infty}(H)$.
In light of Theorem \ref{cor4.1}, one has  a number $m_0$,
a  set of points ${\bf y}=\{y_1, \ldots, y_{m_0}\}\subset\Omega$, and a uniform probability
measure $\nu_0:=\frac{1}{m_0}\sum\limits_{j=1}^{m_0}\delta_{y_j}$ on this set such that
\begin{equation}\label{init-discr}
\frac{1}{2}\|f\|_{L^\Phi(\mu)}\le \|f\|_{L^\Phi(\nu_0)}\le \frac{3}{2}\|f\|_{L^\Phi(\mu)}
\quad \forall f\in X_N.
\end{equation}
On the other hand, by Lemma \ref{lem5.1}, there is a basis $v_1, \ldots, v_N$ in $X_N$
 such that
$$
\int_{{\bf y}'} \varphi\bigl(F(y)\bigr)
F(y)^{-1}
v_{r}(y)v_{r'}(y)\, \nu_0(dy) = c\delta_{r, r'},\quad c>0,
$$
where $F(y):=\bigl(\sum\limits_{k=1}^{N}|v_k(y)|^2\bigr)^{1/2}$ and
${\bf y}'=\{y\in {\bf y}\colon F(y)>0\}$.
In particular,
$$
\int_{{\bf y}'} \varphi\bigl(F(y)\bigr)F(y)\, \nu_0(dy)
=\sum_{r=1}^{N}\int_{{\bf y}'} \varphi\bigl(F(y)\bigr) F(y)^{-1} |v_{r}(y)|^2\, \nu_0(dy)
=cN.
$$
Further, the condition $\Phi\in {\rm (Inc)}_p\cap{\rm (Dec)}$  (cf. Remark \ref{rem-diff})
implies
\begin{equation}\label{eq-der-b}
p\Phi(F(y))\le \varphi\bigl(F(y)\bigr)F(y) \le q\Phi(F(y))
\end{equation}
with some $q:=q(\Phi)\ge 1$ such that $\Phi\in {\rm (Dec)}_q$.
Moreover, Lemma \ref{lem5.1} yields
$$
\int_{{\bf y}} \Phi\bigl(F(y)\bigr)\, \nu_0(dy)=1.
$$ 
Thus, we arrive at the condition
\begin{equation}\label{eq5.2}
1\le p\le cN\le q.
\end{equation}
We set $\displaystyle\tilde{\nu}_0 := \frac{\varphi(F)F}{cN}\nu_0 =
\frac{1}{m_0}\sum\limits_{j=1}^{m_0}\frac{\varphi(F(y_j))F(y_j)}{cN}\delta_{y_j}$,
$\displaystyle \tilde{v}_r:=\frac{\sqrt{N}}{F}v_r$, and
$$\tilde{X}_N:=\{\tilde{f}=F^{-1}f\colon f\in X_N\}.$$
Then we note that
$$
\int_{{\bf y}'}\tilde{v}_r\tilde{v}_{r'}\, d\tilde{\nu}_0
= c^{-1}\int_{{\bf y}'}\varphi(F)F^{-1}v_rv_{r'}\, d\nu_0=\delta_{r, r'}
$$
and
$$
\sum_{r=1}^{N}|\tilde{v}_r|^2\le N,
$$
i.e., $\tilde{X}_N\in NI_{2, \infty}(\sqrt{N})$, where the $L^2$ norm is taken with respect to the measure  $\tilde{\nu}_0$.

By \eqref{cond} with $t=1$, $\Phi(s)\le K_{\Phi, \Psi}\Psi(s)$ and therefore,
$$
\Phi\bigl(N^{\frac{1}{\min\{p, 2\}}}\bigr)
\le K_{\Phi, \Psi}\Psi\bigl(N^{\frac{1}{\min\{p,2\}}}\bigr).
$$
In light of Remark \ref{rem2.1} (with $\delta=1/2$), we  simultaneously discretize
$L^\Phi(\tilde{\nu}_0)$ and $L^\Psi(\tilde{\nu}_0)$ norms on the subspace $\tilde{X}_N$ as follows:
there is a set ${\bf x}:=\{x_1, \ldots, x_m\}\subset {\bf y}'\subset \Omega$ of cardinality
$$
m\le c(\Phi, \Psi, p)
\Psi\bigl(N^{\frac{1}{\min\{p, 2\}}}\bigr)
(\log 2N)^3
$$
such that for a uniform measure $\nu'=\frac{1}{m}\sum\limits_{j=1}^{m}\delta_{x_j}$ on ${\bf x}$ one has
\begin{equation}\label{eq5.8}
\frac{1}{2}\|\tilde{f}\|_{L^\Phi(\tilde{\nu}_0)}\le \|\tilde{f}\|_{L^\Phi(\nu')}\le
\frac{3}{2}\|\tilde{f}\|_{L^\Phi(\tilde{\nu}_0)}
\quad \forall \tilde{f}\in \tilde{X}_N
\end{equation}
and
\begin{equation}\label{eq5.9}
a_{\Psi}(1)^{-1}\frac{1}{2}\|\tilde{f}\|_{L^\Psi(\tilde{\nu}_0)}\le \|\tilde{f}\|_{L^\Psi(\nu')}\le
a_{\Psi}(1)\frac{3}{2}\|\tilde{f}\|_{L^\Psi(\tilde{\nu}_0)}
\quad \forall \tilde{f}\in \tilde{X}_N.
\end{equation}
Setting
\begin{equation}\label{eq5.5}
M_0:=q\max\bigl\{1, \max\bigl\{\Phi(F(y_j))\Phi(\tfrac{1}{F(y_j)})\colon y_j\in{\bf y}'\bigr\}\bigr\},
\end{equation}
we note that $1\le M_0<\infty$ due to the assumptions on the function $\Phi$.
Then, by \eqref{eq-est}, \eqref{eq-der-b}, \eqref{eq5.2},
and the definition of $\tilde{\nu}_0$, one has
$$
\int_{{\bf y}'}\Phi\bigl(M_0^{-1} \tilde{1}\bigr)\, d\tilde{\nu}_0
= (cN)^{-1}\int_{{\bf y}'}\Phi\bigl(M_0^{-1}\tilde{1}\bigr) \varphi(F)F\, d \nu_0
\le \\
q M_0^{-1}
\int_{{\bf y}'}\Phi\bigl(\tfrac{1}{F}\bigr) \Phi(F)\, d \nu_0
\le1,
$$
where $\tilde{1}:=1/F$.
This and
$\|\tilde{1}\|_{L^\Phi(\nu')}\le \frac{3}{2}\|\tilde{1}\|_{L^\Phi(\tilde{\nu}_0)}<2M_0$, cf. \eqref{eq5.8},
imply that
\begin{equation}\label{eq5.6}
\frac{1}{m}\sum\limits_{j=1}^{m}\Phi(\tfrac{1}{2M_0 F(x_j)})\le 1.
\end{equation}
Now we set
\begin{equation}\label{eq5.7}
\lambda_j:=\frac{\Phi(\max\{\frac{1}{2M_0 F(x_j)}, 1\})}{\sum_{k=1}^{m}\Phi(\max\{\frac{1}{2M_0F(x_k)}, 1\})}
\end{equation}
and
$\nu = \sum\limits_{j=1}^{m}\lambda_j\delta_{x_j}$.

Fix $f\in X_N$ with $\|f\|_{L^\Psi(\nu)}< 1$.
Let us prove that
\begin{equation}\label{eq-new-1}
\|\tilde{f}\|_{L^\Psi(\nu')}\le C_1(\Phi, \Psi, p)\bigl(1+\Psi\circ\Phi^{-1}(1) + M_{\Phi, \Psi}(m)\bigr)^{1/p}.
\end{equation}
First,
by  \eqref{eq-est}, \eqref{cond},
and  \eqref{eq5.6}, we obtain
\begin{eqnarray*}
\int_{{\bf x}}\Phi\bigl(\tfrac{1}{K_{\Phi, \Psi} (1+\Phi(1))}(2M_0)^{-1}|\tilde{f}|\bigr)\, d\nu'
&\le& \frac{1}{K_{\Phi, \Psi} (1+\Phi(1))} \int_{{\bf x}}\Phi\bigl((2M_0)^{-1}|\tilde{f}|\bigr)\, d\nu'
\\
&=&
\frac{1}{K_{\Phi, \Psi} (1+\Phi(1))}\cdot\frac{1}{m}\sum_{j=1}^{m}\Phi\bigl((2M_0)^{-1}|\tilde{f}(x_j)|\bigr)
\\
&\le&
\frac{1}{1+\Phi(1)}\cdot
\frac{1}{m}\sum_{j=1}^{m}\Psi(|f(x_j)|)\Phi\bigl(\max\{\tfrac{1}{2M_0 F(x_j)}, 1\}\bigr)
\\
&\le&
\frac{1}{1+\Phi(1)}\cdot
\frac{1}{m}\Bigl(\sum_{k=1}^{m}\Phi\bigl(\max\{\tfrac{1}{2M_0F(x_k)}, 1\}\bigr)\Bigr)
\int_{{\bf x}}\Psi(|f|)\, d\nu
\\
&\le&
\frac{1}{1+\Phi(1)}\cdot
\frac{1}{m}\sum_{k=1}^{m}\Phi\bigl(\max\{\tfrac{1}{2M_0F(x_k)}, 1\}\bigr)\le 1,
\end{eqnarray*}
where $M_0$ is defined by \eqref{eq5.5}.
Therefore,
$\|\tilde{f}\|_{L^\Phi(\nu')}\le 2M_0K_{\Phi, \Psi} (1+\Phi(1))<3M_0K_{\Phi, \Psi} (1+\Phi(1))$, which implies the estimate
$$
\Phi\Bigl(\frac{\max_{1\le j\le n}|\tilde{f}(x_j)|}{3M_0K_{\Phi, \Psi} (1+\Phi(1))}\Bigr)
=\max_{1\le j\le n}\Phi\Bigl(\frac{|\tilde{f}(x_j)|}{3M_0K_{\Phi, \Psi} (1+\Phi(1))}\Bigr)
\le m.
$$
Second,
letting  $M_1=3M_0K_{\Phi, \Psi} (1+\Phi(1))$ and $M_2:= 1+\Psi\circ\Phi^{-1}(1) + M_{\Phi, \Psi}(m)$, from \eqref{eq-est-p}
we derive that
\begin{eqnarray*}
&&\int_{{\bf x}}\Psi\bigl(a_\Psi(p)^{-1/p}M_2^{-1/p}M_1^{-1}|\tilde{f}|\bigr)\, d\nu'
\le
M_2^{-1} \int_{\{M_1^{-1}|\tilde{f}|\le \Phi^{-1}(1)\}}\Psi\bigl(M_1^{-1}|\tilde{f}|\bigr)\, d\nu'
\\&&\qquad\qquad\qquad\qquad\qquad\qquad\qquad\quad+
M_2^{-1}
\int_{\{M_1^{-1}|\tilde{f}|\ge \Phi^{-1}(1)\}}\Psi\bigl(M_1^{-1}|\tilde{f}|\bigr)\, d\nu'
 \\
&\le& M_2^{-1}\Bigl(\Psi\circ\Phi^{-1}(1) + M_{\Phi, \Psi}(m)
\int_{\{M_1^{-1}|\tilde{f}|\ge \Phi^{-1}(1)\}}\Phi\bigl(M_1^{-1}|\tilde{f}|\bigr)\, d\nu'
\Bigr)\le1,
\end{eqnarray*}
which yields  $\|\tilde{f}\|_{L^\Psi(\nu')}\le a_\Psi(p)^{1/p}M_1M_2^{1/p}$, i.e., inequality \eqref{eq-new-1} holds.

Then, using  the left-hand side bound in \eqref{eq5.9}, we get
\begin{equation}\label{eq-vsp2}
\|\tilde{f}\|_{L^\Psi(\tilde{\nu}_0)}\le 2a_\Psi(1)\|\tilde{f}\|_{L^\Psi(\nu')}\le
2a_\Psi(1) a_\Psi(p)^{1/p}M_1M_2^{1/p}.
\end{equation}
By \eqref{eq-est}, \eqref{eq-der-b},
the definition of $\tilde{\nu}_0$,
relation \eqref{eq5.2} and property \eqref{cond}, one has
\begin{eqnarray*}
&&\int_{{\bf y}}\Phi\bigl((qK_{\Phi, \Psi})^{-1}(3a_\Psi(1)a_\Psi(p)^{1/p}M_1M_2^{1/p})^{-1}|f|\bigr)\, d\nu_0
\\
&\le&
(qK_{\Phi, \Psi})^{-1}
\int_{{\bf y}}\Phi\bigl((3a_\Psi(1)a_\Psi(p)^{1/p}M_1M_2^{1/p})^{-1}|f|\bigr)\, d\nu_0
\\
&\le&
q^{-1}
\int_{{\bf y}'}\Psi\bigl((3a_\Psi(1)a_\Psi(p)^{1/p}M_1M_2^{1/p})^{-1}|\tilde{f}|\bigr)\Phi(F)\, d\nu_0
\\
&\le&
q^{-1}
\int_{{\bf y}'}\Psi\bigl((3a_\Psi(1)a_\Psi(p)^{1/p}M_1M_2^{1/p})^{-1}|\tilde{f}|\bigr)\varphi(F)F\, d\nu_0
\\
&=& q^{-1}cN
\int_{{\bf y}'}\Psi\bigl((3a_\Psi(1)a_\Psi(p)^{1/p}M_1M_2^{1/p})^{-1}|\tilde{f}|\bigr)\, d\tilde{\nu}_0
\\
&\le&
\int_{{\bf y}'}\Psi\bigl((3a_\Psi(1)a_\Psi(p)^{1/p}M_1M_2^{1/p})^{-1}|\tilde{f}|\bigr)\, d\tilde{\nu}_0
\le 1,
\end{eqnarray*}
where the last inequality follows from \eqref{eq-vsp2}.
This implies
\begin{eqnarray*}
\|f\|_{L^\Phi(\nu_0)}&\le& C_2(\Phi, \Psi, p)(1+\Psi\circ\Phi^{-1}(1) + M_{\Phi, \Psi}(m))^{1/p}
\\&\le& 3^{1/p} C_2(\Phi, \Psi, p) \bigl(M_{\Phi, \Psi}(m)\bigr)^{1/p},
\end{eqnarray*}
where
$$
C_2(\Phi, \Psi, p):=9qK_{\Phi, \Psi}^2a_\Psi(1)a_\Psi(p)^{1/p}(1+\Phi(1))M_0.
$$
Finally,  \eqref{init-discr} yields
$$
\|f\|_{L^\Phi(\mu)}\le 2\|f\|_{L^\Phi(\nu_0)}\le
C_3(\Phi, \Psi, p) \bigl(M_{\Phi, \Psi}(m)\bigr)^{1/p}
$$
with $C_3(\Phi, \Psi, p):=2\cdot 3^{1/p}C_2(\Phi, \Psi, p)$.
The proof of the theorem is now complete.
\end{proof}

We note that in Theorem \ref{t5.1} one can actually choose  weights $\lambda_j$ to be equal.

\begin{remark}\label{rem5.1}
{\rm
Under the same assumptions on the functions $\Phi$ and $\Psi$ as in Theorem~\ref{t5.1},
we obtain the following statement:
There are positive constants $C:=C(\Phi, \Psi, p)$ and $c:=c(\Phi, \Psi, p)$
such that
for every $N$-dimensional subspace $X_N\subset C(\Omega)$
there is a set ${\bf x}:=\{x_1, \ldots, x_m\}\subset \Omega$
of cardinality
$$
m\le c\Psi\bigl(N^\frac{1}{\min\{p,2\}}\bigr)(\log 2N)^3
$$
for which
$$
\|f\|_\Phi\le C\bigl(M_{\Phi, \Psi}(m)\bigr)^{1/p}
\|f\|_{\Psi, {\bf x}} \quad \forall f\in X_N.
$$

Indeed, if $1\not\in X_N$, we consider the space $X_N':={\rm span}\{X_N, 1\}$ of dimension $N+1\le 2N$.
We now argue as in the proof of Proposition 3.1 from \cite{LMT24}.
Let $\{y_1, \ldots, y_k\}$ and $\{\lambda_1, \ldots, \lambda_k\}$, $\sum_{j=1}^{k}\lambda_j=1$, be the points
and weights from Theorem \ref{t5.1} with
$$
k\le c\Psi\bigl((2N)^\frac{1}{\min\{p,2\}}\bigr)(\log 4N)^3
$$
and
$$
\|f\|_\Phi\le C\bigl(M_{\Phi, \Psi}(k)\bigr)^{1/p}
\|f\|_{\Psi, {\bf y}, \lambda} \quad \forall f\in X_N'.
$$
There holds
$$
m:=\sum_{j=1}^k([\lambda_jk]+1)
\le
2k
\le 16c 2^\frac{q}{\min\{p,2\}}\Psi\bigl(N^\frac{1}{\min\{p,2\}}\bigr)(\log 2N)^3,
$$
where $q:=q(\Psi)$ is a number such that $\Psi\in {\rm (Dec)}_q$.
Take points $\{x_1, \ldots, x_m\}$ such that
$([\lambda_jk]+1)$ of them coincide with $y_j$ for each $j\in\{1, \ldots, k\}$.
Let now $f\in X_N'$ be such that $\|f\|_{\Psi, {\bf x}}<1$. Then, by \eqref{eq-est},
\begin{eqnarray*}
\sum_{j=1}^{k}\lambda_j\Psi\bigl(a_\Psi(1)^{-1}2^{-1}|f(y_j)|\bigr)
&\le& 2^{-1}\sum_{j=1}^{k}\lambda_j\Psi\bigl(|f(y_j)|\bigr)
\le \frac{1}{2k}\sum_{j=1}^{k}(\lambda_jk)\Psi\bigl(|f(y_j)|\bigr)
\\
&\le& \frac{1}{m}\sum_{j=1}^{k}([\lambda_jk]+1)\Psi\bigl(|f(y_j)|\bigr)
\le \frac{1}{m}\sum_{i=1}^{m}\Psi\bigl(|f(x_i)|\bigr)\le 1.
\end{eqnarray*}
Since $k\le m$, we have
$$
\|f\|_\Phi\le C \bigl(M_{\Phi, \Psi}(k)\bigr)^{1/p}
\|f\|_{\Psi, {\bf y}, \lambda}
\le 2a_\Psi(1) C\bigl(M_{\Phi, \Psi}(m)\bigr)^{1/p}
$$
completing the proof.
}
\end{remark}

\begin{remark}\label{rem5.2}
{\rm
We point out that one can always take
$$
\Psi_\Phi(s):= \sup_{t>0}\frac{\Phi(ts)}{\Phi(t)}
$$
so that condition \eqref{cond} holds with $K_{\Phi, \Psi_\Phi}=1$.
Indeed, for a function $\Phi\in {\rm (aInc)}_p$ (respectively, $\Phi\in{\rm (aDec)}_q$)
one always has $\Psi_\Phi\in {\rm (aInc)}_p$ ($\Psi_\Phi\in{\rm (aDec)}_q$)
with $a_{\Psi_\Phi}(p) = a_{\Phi}(p)$ ($b_{\Psi_\Phi}(q) = b_{\Phi}(q)$).
}
\end{remark}

We now apply Theorem \ref{t5.1} and Remark \ref{rem5.1} to the function
$\Phi_{p, \alpha, \beta} := t^p\frac{(\ln(e+t))^\alpha}{(\ln(e+t^{-1}))^\beta}$, cf.
\eqref{eq-Phi}.

\begin{example}\label{ex5.1}
{\rm
Let $p\ge2$ and $\alpha\ge0$.
There are positive constants $c:=c(p, \alpha)$ and
$C:=C(p, \alpha)$
such that, for every $N$-dimensional space $X_N\subset C(\Omega)$,
there exists a set ${\bf x}:=\{x_1, \ldots, x_m\}\subset \Omega$
of cardinality
$$
m\le CN^{\frac{p}{\min\{p, 2\}}}(\log 2N)^{3+2\alpha}
$$
for which
$$
\|f\|_{\Phi_{p, \alpha, \alpha}}\le c(\log 2N)^{\alpha/p}
\|f\|_{\Phi_{p, 2\alpha, 0}, {\bf x}} \quad \forall f\in X_N.
$$
}
\end{example}
Indeed, we have
$$\Phi_{p, \alpha, \alpha}(ts)\le  4^\alpha\Phi_{p, \alpha, \alpha}(t)\Phi_{p, 2\alpha, 0}(s), \qquad t,s>0,$$
i.e., condition \eqref{cond} holds with $\Psi=\Phi_{p, 2\alpha, 0}$.
Moreover,
$$
\sup\limits_{t>0}\bigr(\Phi_{p, \alpha, \alpha}(t)\Phi_{p, \alpha, \alpha}(t^{-1})\bigl)
= 1
$$
and, for every $m\ge 1$,
\begin{equation*}\label{eq-vsp4}
M_{\Phi_{p, \alpha, \alpha}, \Phi_{p, 2\alpha, 0}}(m)
:=
\max\Bigl\{1, \max\Bigl\{\frac{\Phi_{p, 2\alpha, 0}(t)}{\Phi_{p, \alpha, \alpha}(t)}\colon
t\in[\Phi_{p, \alpha, \alpha}^{-1}(1), \Phi_{p, \alpha, \alpha}^{-1}(m)]\Bigr\}\Bigr\}\le 2^\alpha (\ln(e+m))^\alpha.
\end{equation*}
To see the last estimate, we note that
$\displaystyle
\frac{\Phi_{p, 2\alpha, 0}(t)}{\Phi_{p, \alpha, \alpha}(t)}
= (\ln(e+t)\ln(e+t^{-1}))^\alpha\le 2^\alpha (\ln(e+t))^\alpha
$ and
$\Phi_{p, \alpha, \alpha}\in {\rm (Inc)}_1$.
Finally, we apply  \eqref{result} and Remark \ref{rem5.1} to complete the proof.

\subsection{Bounds of integral Orlicz norms by discrete $L^2$-norms}
\label{ssect6-2}

In this subsection we are going to prove the following general one-sided discretization
result.

\begin{theorem}\label{T-1sided}
Let $\Phi\in {\bf \Phi}_w$.
Assume that, for every $N\ge 1$, there is a number $K:=K(N, \Phi)>0$ such that
$$
\sup\limits_{0<t\le \sqrt{N}}\frac{\Phi\bigl(K^{-1}t\bigr)}{t^2} <1.
$$
Then there are positive numbers $c_1$ and $c_2$ such that,
for every $N$-dimensional subspace $X_N\subset C(\Omega)$,
there is a set ${\bf x}:=\{x_1, \ldots, x_m\}\subset \Omega$
of cardinality $m\le c_1N$ for which
$$
\|f\|_\Phi\le c_2K \|f\|_{2, {\bf x}} \quad \forall f\in X_N.
$$
\end{theorem}

The proof of this theorem  follows the ideas from \cite{KPUU24}, cf.  Theorem 13 there.
In particular, we will use the following two lemmas.

\begin{lemma}[see Lemma 11 in \cite{KPUU24} or Proposition 3.2 in
 \cite{LMT24}]\label{lem-discr-1}
There exist two numerical universal constants $c_1, c_2\ge 1$ such that for any $N\ge 1$ and
for any $N$-dimensional subspace $X_N\subset C(\Omega)$, there is a number $m:=m(N)\in[N, c_1N]$ and
 points $y_1, \ldots, y_{m}\in \Omega$ such that
$$
\|f\|_2\le c_2 \Bigl(\frac{1}{m}\sum_{j=1}^{m}|f(y_j)|^2\Bigr)^{1/2}\quad \forall f\in X_N.
$$
\end{lemma}

\begin{lemma}[see Theorem 12 in \cite{KPUU24}]\label{lem-discr-2}
There are two numerical universal constants $c_1, c_2\ge 1$ such that for any $N\ge 1$ and
for any $N$-dimensional subspace $X_N\subset C(\Omega)$, there is a number $m:=m(N)\in[N, c_1N]$ and
points $z_1, \ldots, z_{m}\in \Omega$ such that
$$
\|f\|_\infty \le c_2 \sqrt{N} \Bigl(\frac{1}{m}\sum_{j=1}^{m}|f(z_j)|^2\Bigr)^{1/2}\quad \forall f\in X_N.
$$
\end{lemma}

We point out that Lemma
\ref{lem-discr-2}
follows from Lemma
\ref{lem-discr-1}
 taking into account the following useful result by Kiefer and Wolfowitz
(see also the discussion in \cite{KPUU24}). 

\begin{theorem}[\cite{KW}]
Let $N\ge1$. For any $N$-dimensional subspace $X_N\subset C(\Omega)$,
there exists a probability measure $\mu$ on $\Omega$ such that
$$
\|f\|_\infty \le \sqrt{N}\|f\|_{L_2(\Omega,\mu)}\quad \forall f\in X_N.
$$
\end{theorem}

We are now in a position  to prove Theorem \ref{T-1sided}.

{\bf Proof of Theorem \ref{T-1sided}.}
For $N\ge 1$, let $m:=m(N)$ be the number of points sufficient for the successful discretization provided by Lemmas \ref{lem-discr-1} and \ref{lem-discr-2} and
 ${\bf x}\!:=\!\{y_1, \ldots, y_{m}, z_1, \ldots, z_{m}\}$ be the union of the sets of points from these lemmas.
We note that, for every $f\in X_N$ with $\|f\|_{2, {\bf x}}\le 1$, one has
$$
\|f\|_\infty\le c_2\sqrt{N} \Bigl(\frac{1}{m}\sum_{j=1}^{m}|f(z_j)|^2\Bigr)^{1/2}\le
c_2\sqrt{N}\sqrt{2}\Bigl(\frac{1}{2m}\sum_{j=1}^{2m}|f(x_j)|^2\Bigr)^{1/2}
\le C \sqrt{N}
$$
and
$$
\|f\|_2\le c_2 \Bigl(\frac{1}{m}\sum_{j=1}^{m}|f(y_j)|^2\Bigr)^{1/2}\le
c_2\sqrt{2}\Bigl(\frac{1}{2m}\sum_{j=1}^{2m}|f(x_j)|^2\Bigr)^{1/2}\le C.
$$
For any $f\in X_N$ with $\|f\|_{2, {\bf x}}\le 1$, we have
\begin{eqnarray*}
&&\int_{\Omega}\Phi\bigl(K^{-1}C^{-1}|f(x)|\bigr)\, \mu(dx)
=\int_{\Omega}\frac{\Phi\bigl(K^{-1}C^{-1}|f(x)|\bigr)}{C^{-2}|f(x)|^2} |f(x)|^2C^{-2}\, \mu(dx)
\\
&\le& \sup\limits_{0<t\le \sqrt{N}}\frac{\Phi\bigl(K^{-1}t\bigr)}{t^2}C^{-2}\int_{\Omega} |f(x)|^2\, \mu(dx)
\le \sup\limits_{0<t\le \sqrt{N}}\frac{\Phi\bigl(K^{-1}t\bigr)}{t^2}<1.
\end{eqnarray*}
Thus,
$$
\|f\|_\Phi\le CK \|f\|_{2, {\bf x}}
\quad \forall f\in X_N.
$$
The proof is now complete.
\qed

\vspace{3mm}
For ${\bf \Phi}$-functions $\Phi \in {\rm (aInc)}_p$, $p\ge 2$,
we obtain the following more explicit version of Theorem~\ref{T-1sided}.

\begin{corollary}\label{T-1sided-1}
Let $p\in [2, \infty)$, $N\ge 1$.
Let $\Phi$ be ${\bf \Phi}$-functions such that
$\Phi \in {\rm (aInc)}_p$.
There are positive numbers $c_1$ and $c_2$, where
$c_2=c_2(\Phi, p)$ depends
 only on $\Phi$ and  $p$,
such that
for every $N$-dimensional subspace $X_N\subset C(\Omega)$,
there is a set ${\bf x}:=\{x_1, \ldots, x_m\}\subset \Omega$
of cardinality $m\le c_1N$ for which
$$
\|f\|_\Phi\le c_2\Bigl(\frac{\Phi\bigl(\sqrt{N}\bigr)}{N}\Bigr)^{1/p}\|f\|_{2, {\bf x}} \quad \forall f\in X_N.
$$
Moreover,
one can take $c_2(\Phi, p)$ to be of the form
$$
c_2(\Phi, p)=C a_\Phi(p)^{2/p}\bigl(1+ (\Phi(1))^{-1}\bigr)^{1/p}
$$
where $C\ge 1$ is a numerical constant.
\end{corollary}

\begin{proof}
Let $M= a_\Phi(p)\frac{\Phi(\sqrt{N})}{N}\cdot \frac{1+ \Phi(1)}{\Phi(1)}\ge 1$.
Then, by \eqref{eq-est-p}, for any $t\in(0, \sqrt{N}]$, one has
$$
\frac{\Phi\bigl(a_\Phi(p)^{-1/p}M^{-1/p}t\bigr)}{t^2}
\le M^{-1} \frac{\Phi\bigl(t\bigr)}{t^2} = M^{-1} \frac{\Phi\bigl(t\bigr)}{t^p} t^{p-2}
\le  M^{-1} a_\Phi(p)\frac{\Phi\bigl(\sqrt{N}\bigr)}{N}<1.
$$
Taking into account Theorem \ref{T-1sided} with $K = a_\Phi(p)^{1/p}M^{1/p}$, we obtain
$$
\|f\|_\Phi\le Ca_\Phi(p)^{1/p}M^{1/p} \|f\|_{2, {\bf x}}
\le
Ca_\Phi(p)^{2/p}\bigl(1+ (\Phi(1))^{-1}\bigr)^{1/p}\Bigl(\frac{\Phi(\sqrt{N})}{N}\Bigr)^{1/p} \|f\|_{2, {\bf x}}
\quad \forall f\in X_N.
$$
The proof is now complete.
\end{proof}

Recall that  $\displaystyle\Phi_{p, \alpha, \beta}(t)= t^p\frac{(\ln(e+t))^\alpha}{(\ln(e+t^{-1}))^\beta}$, $p\ge 1$, $\alpha, \beta\ge 0$,
cf. \eqref{eq-Phi}. Note that $\Phi \in {\rm (Inc)}_p$, $a_\Phi(p)=1$, and $\Phi(1)=1$.
In addition,
$$
\Phi_{p, \alpha, \beta}(\sqrt{N}) \le N^{p/2} (\log 4N)^{\alpha}.
$$
Thus, Corollary \ref{T-1sided-1} yields the following result.

\begin{example}\label{cor5.1}{\rm
Let 
$p\ge 2$, $\alpha, \beta\ge 0$, 
and  $N\ge1$.
There exist numerical constants $c_1, c_2>0$ satisfying the following condition:  for every $N$-dimensional space $X_N\subset C(\Omega)$,
there exist a set ${\bf x}:=\{x_1, \ldots, x_m\}\subset \Omega$ of cardinality
$
m\le c_1N
$
such that
$$
\|f\|_{\Phi_{p, \alpha, \beta}}\le c_2N^{\frac{1}{2}-\frac{1}{p}}(\log 4N)^{\frac{\alpha}{p}}
\|f\|_{2, {\bf x}} \quad \forall f\in X_N.
$$
In the case $\alpha = \beta= 0$ we recover the result of Theorem 13 in \cite{KPUU24}.
}
\end{example}

By applying Theorem \ref{T-1sided} directly, we obtain the following one-sided discretization result for exponential Orlicz functions.

\begin{example}\label{cor5.2}{\rm
Let $q\ge 2$, $\Phi_q(t):= e^{t^q} - 1$.
There are positive numbers $c_1$ and $c_2$ such that,
for every $N$-dimensional subspace $X_N\subset C(\Omega)$,
there is a set ${\bf x}:=\{x_1, \ldots, x_m\}\subset \Omega$
of cardinality $m\le c_1N$ for which
$$
\|f\|_{\Phi_q}\le c_2\frac{N^{1/2}}{(\log(N+1))^{1/q}} \|f\|_{2, {\bf x}} \quad \forall f\in X_N.
$$
}
\end{example}

\begin{proof}
Note that the function $t\mapsto \frac{\Phi_q(t)}{t^2}$ is increasing on $[0, +\infty)$
and
$$
\sup\limits_{0<t\le \sqrt{N}}\frac{\Phi_q\bigl(K^{-1}t\bigr)}{t^2} = \frac{\Phi_q\bigl(K^{-1}\sqrt{N}\bigr)}{N}
= \frac{e^{K^{-q}N^{q/2}} - 1}{N}.
$$
Letting $K = \frac{2N^{1/2}}{(\log(N+1))^{1/q}}$ implies
$
\frac{e^{K^{-q}N^{q/2}} - 1}{N} 
=\frac{(N+1)^{2^{-q}\log e} - 1}{N}<1.
$
Theorem \ref{T-1sided} now yields  the announced discretization inequality.
\end{proof}

\vspace{0.2cm}

\section{Applications to  sampling recovery}\label{sect7}
In this section, we 
 obtain new estimates of the sampling numbers in the Orlicz norm in terms of the corresponding
Kolmogorov widths in the uniform norm.
Let ${\bf F}$ be subset of some function Banach space $(L, \|\cdot\|)$.
We further always assume that ${\bf F}\subset C(\Omega)$ for some compact set $\Omega$.
The sampling numbers of a function class ${\bf F}$ are defied by 
$$
\varrho_m({\bf F}, \|\cdot\|):=
\inf_{\substack{{\bf x}\subset \Omega\\ \# {\bf x}\le m}}
\inf_{ T_{{\bf x}} - \hbox{ linear}}
\sup\limits_{f\in {\bf F}}
\|f-T_{{\bf x}}(f(x_1), \ldots, f(x_m))\|.
$$
Similarly, we introduce the modified sampling numbers of a function class ${\bf F}$ as
$$
\varrho_m^*({\bf F}, \|\cdot\|):= \inf_{\substack{{\bf x}\subset \Omega\\ \# {\bf x}\le m}}
\inf_{X_N, N\le m}\inf_{T_{{\bf x}}\colon \mathbb{C}^m\to X_N}
\sup\limits_{f\in {\bf F}}\|f-T_{{\bf x}}(f(x_1), \ldots, f(x_m))\|.
$$
Let
$$
d_N({\bf F}, \|\cdot\|)=
\inf_{X_N}\sup_{f\in{\bf F}}
\inf_{u\in X_N}\|f-u\|
$$
be the $N$-th Kolmogorov width of ${\bf F}$.
For a fixed $N$-dimensional subspace $X_N$ in $C(\Omega)$
we consider the following error of approximation of a function by elements of $X_N$
in the uniform norm:
$$
d(f, X_N)_\infty:=\inf_{u\in X_N}\|f-u\|_\infty.
$$
For a fixed number $m\in \mathbb{N}$, a set of points
${\bf x}=\{x_1, \ldots, x_m\}\subset \Omega$,
and a probability measure
$\nu=\sum_{j=1}^{m}\lambda_j\delta_{x_j}$ on ${\bf x}$,
we consider the following recovery algorithm
$$
\ell_{L^\Phi(\nu), X_N}(f):=
{\rm arg}\min\limits_{u\in X_N}\|f-u\|_{L^\Phi(\nu)}.
$$
based on functional values $(f(x_1), \ldots, f(x_m))$.

We need the following extension of Theorem 2.1 in \cite{Tem21} to the Orlicz setting.

\begin{lemma}\label{lem7.1} Let $p\ge 1$, $D\ge 1$, and let $\Phi, \Psi\in {\rm (aInc)}_p$  be a pair of ${\bf\Phi}$-functions.
Then for any
probability measure $\mu$ on $\Omega$,
any discrete probability measure $\nu$ on
a finite subset of $\Omega$, and
for any
$N$-dimensional subspace $X_N\subset C(\Omega)$, satisfying
$$
\|u\|_{L^\Phi(\mu)}\le D\|u\|_{L^\Psi(\nu)} \quad \forall u\in X_N,
$$
there holds
$$
\|f - \ell_{L^\Psi(\nu), X_N}(f)\|_{L^\Phi(\mu)}
\le C(\Phi, \Psi, p, D) d(f, X_N)_\infty,
$$
where
\begin{equation}\label{eq-const}
C(\Phi, \Psi, p, D)=C_\Phi\bigl(2a_\Phi(p)^{1/p}(\Phi(1) + 1)^{1/p} + 4DC_\Psi a_\Psi(p)^{1/p}(\Psi(1) + 1)^{1/p}\bigr)
\end{equation}
and $C_\Psi$ and $C_\Phi$ are the constants from inequality \eqref{eq-triangle}.
\end{lemma}

\begin{proof}
We follow the argument from 
 \cite{Tem21}.
Let $P_{X_N}(f)\in X_N$ be such that
$$
\|f-P_{X_N}(f)\|_\infty\le 2d(f, X_N)_\infty.
$$
Then, by \eqref{eq-est-p}, there holds
$$
\|f-P_{X_N}(f)\|_{L^\Phi(\mu)}\le
a_\Phi(p)^{1/p}(\Phi(1) + 1)^{1/p}\|f-P_{X_N}(f)\|_\infty\le
2a_\Phi(p)^{1/p}(\Phi(1) + 1)^{1/p}d(f, X_N)_\infty
$$
and, similarly,
$$
\|f-P_{X_N}(f)\|_{L^\Psi(\nu)}
\le
2a_\Psi(p)^{1/p}(\Psi(1) + 1)^{1/p}d(f, X_N)_\infty.
$$
Then it is clear that
$$
\|f-\ell_{L^\Psi(\nu), X_N}(f)\|_{L^\Psi(\nu)}
\le \|f-P_{X_N}(f)\|_{L^\Psi(\nu)}
\le 2a_\Psi(p)^{1/p}(\Psi(1) + 1)^{1/p}d(f, X_N)_\infty.
$$
Taking this into account, we get
$$
\|P_{X_N}(f)-\ell_{L^\Psi(\nu), X_N}(f)\|_{L^\Psi(\nu)}
\le 4C_\Psi a_\Psi(p)^{1/p}(\Psi(1) + 1)^{1/p}d(f, X_N)_\infty
$$
and
\begin{eqnarray*}
\|P_{X_N}(f)-\ell_{L^\Psi(\nu), X_N}(f))\|_{L^\Phi(\mu)}
&\le&
D\|P_{X_N}(f)-\ell_{L^\Psi(\nu), X_N}(f)\|_{L^\Psi(\nu)}
\\
&\le& 4DC_\Psi a_\Psi(p)^{1/p}(\Psi(1) + 1)^{1/p}d(f, X_N)_\infty.
\end{eqnarray*}
Thus  we arrive at
\begin{eqnarray*}
\|f - \ell_{L^\Psi(\nu), X_N}(f)\|_{L^\Phi(\mu)}
&\le& C_\Phi\bigl(\|f-P_{X_N}(f)\|_{L^\Phi(\mu)} + \|P_{X_N}(f)-\ell_{L^\Psi(\nu), X_N}(f))\|_{L^\Phi(\mu)}\bigr)
\\
&\le&
C(\Phi, \Psi, p, D) d(f, X_N)_\infty,
\end{eqnarray*}
where $C(\Phi, \Psi, p, D)$ is given by \eqref{eq-const}.
\end{proof}

We now prove the Orlicz counterparts of the recent result in \cite{KPUU24} (see  Theorem 20 there).

\begin{theorem}\label{t7.3}
Let $\Phi\in {\bf \Phi}_w$.
Assume that, for every $N\ge 1$, there is a number $K:=K(N, \Phi)>0$ such that
$$
\sup\limits_{0<t\le \sqrt{N}}\frac{\Phi\bigl(K^{-1}t\bigr)}{t^2} <1.
$$
Then there is a positive number $c\ge 1$ such that,
for any probability Borel measure $\mu$ on $\Omega$ and
for any function class ${\bf F}\subset C(\Omega)$,
one has
$$
\varrho_{cN}({\bf F}, \|\cdot\|_{L^\Phi(\mu)})
\le C(N, \Phi)d_N({\bf F}, \|\cdot\|_\infty),
$$
where
$$
C(N, \Phi)=C_\Phi\bigl(2a_\Phi(1)(\Phi(1)+1) + C K(N, \Phi)\bigr),
$$
$C>0$ is a numerical constant and $C_\Phi$ is the constant from inequality \eqref{eq-triangle}.
\end{theorem}

\begin{proof}
Let $X_N$ be such that
$$
\sup_{f\in{\bf F}}d(f, X_N)_\infty
\le
2d_N({\bf F}, \|\cdot\|_\infty).
$$
Take  $c_1$, $c_2$, and
${\bf x}:=\{x_1, \ldots, x_m\}
\subset \Omega$  provided by Theorem \ref{T-1sided} (in particular, $m \le c_1N$).
By Lemma \ref{lem7.1}, for every $f\in{\bf F}$, we have
$$
\|f - \ell_{L^2(\nu), X_N}(f)\|_{L^\Phi(\mu)}
\le   C_\Phi\bigl(2a_\Phi(1)(\Phi(1) + 1) + 8c_2 K(N, \Phi)\bigr) d(f, X_N)_\infty.
$$
It remains to notice that $\ell_{L^2(\nu), X_N}(f)$ is linear in
$(f(x_1), \ldots, f(x_m))$ since it is an orthogonal projection
on $X_N$ in $L^2(\nu)$.
\end{proof}

\begin{example}\label{ex7.3}
{\rm
Let $q\ge 2$, $\Phi_q(t):= e^{t^q} - 1$.
There are numerical constants $c, C\ge1$ such that,
for any probability Borel measure $\mu$ on $\Omega$ and
for any function class ${\bf F}\subset C(\Omega)$,
one has
$$
\varrho_{cN}({\bf F}, \|\cdot\|_{L^{\Phi_q}(\mu)})
\le \frac{CN^{1/2}}{(\log(N+1))^{1/q}}d_N({\bf F}, \|\cdot\|_\infty).
$$
}
\end{example}

We now present a more explicit version of Theorem~\ref{t7.3}
for ${\bf \Phi}$-functions $\Phi \in {\rm (aInc)}_p$, $p\ge 2$.

\begin{theorem}\label{t7.1}
Let $p\in [2, \infty)$, $N\ge 1$.
Let $\Phi$ be a ${\bf \Phi}$-function such that
$\Phi \in {\rm (aInc)}_p$.
There exist a positive numerical constant $c\ge1$ and a number $C(\Phi, p)\ge 1$,
depending only on $\Phi$ and $p$, such that,
for any probability Borel measure $\mu$ on $\Omega$ and
for any function class ${\bf F}\subset C(\Omega)$,
one has
$$
\varrho_{cN}({\bf F}, \|\cdot\|_{L^\Phi(\mu)})
\le C(\Phi, p)\Bigl(\frac{\Phi\bigl(\sqrt{N}\bigr)}{N}\Bigr)^{1/p}d_N({\bf F}, \|\cdot\|_\infty).
$$
Moreover,
one can take $C(\Phi, p)$ to be of the form
$$
C(\Phi, p)=C\cdot C_\Phi a_\Phi(p)^{2/p}\bigl(1+ (\Phi(1))^{-1}\bigr)^{1/p},
$$
where $C\ge 1$ is a numerical constant and $C_\Phi$ is the constant from inequality \eqref{eq-triangle}.
\end{theorem}

\begin{proof}
We argue as in the proof of Theorem \ref{t7.3}.
Let $X_N$ be such that
$$
\sup_{f\in{\bf F}}d(f, X_N)_\infty
\le
2d_N({\bf F}, \|\cdot\|_\infty).
$$
Let $c_1, c_2:=c_2(\Phi, p)>0$, $m \le c_1N$, and
${\bf x}\subset \Omega$ be provided by Corollary \ref{T-1sided-1}.
Recall that one can take
$$
c_2(\Phi, p)=c_0 a_\Phi(p)^{2/p}\bigl(1+ (\Phi(1))^{-1}\bigr)^{1/p},
$$
where $c_0\ge 1$ is a numerical constant.
Using Lemma \ref{lem7.1}, for every $f\in{\bf F}$ we have
\begin{eqnarray*}
\|f - \ell_{L^2(\nu), X_N}(f)\|_{L^\Phi(\mu)}\!\!\!\!
&\le&  \!\!\! C_\Phi\Bigl(2a_\Phi(p)^{1/p}(\Phi(1) + 1)^{1/p} + 8 c_2\Bigl(\frac{\Phi\bigl(\sqrt{N}\bigr)}{N}\Bigr)^{1/p}\Bigr) d(f, X_N)_\infty
\\
&\le&\!\!\!
2C_\Phi\Bigl(2(a_\Phi(p))^{2/p}\bigl(1+ (\Phi(1))^{-1}\bigr)^{1/p} + 8c_2\Bigr) \Bigl(\frac{\Phi\bigl(\sqrt{N}\bigr)}{N}\Bigr)^{1/p}d_N({\bf F}, \|\cdot\|_\infty).
\end{eqnarray*}
Finally, we again notice that $\ell_{L^2(\nu), X_N}(f)$ is linear in
$(f(x_1), \ldots, f(x_m))$ since it is an orthogonal projection
on $X_N$ in $L^2(\nu)$.
\end{proof}

\begin{example}\label{ex7.1}{\rm
Let $p\ge2$, $\alpha, \beta\ge0$, and $\Phi_{p, \alpha, \beta}(t) := t^p\frac{(\ln(e+t))^\alpha}{(\ln(e+t^{-1}))^\beta}$.
There exist positive numerical constants $c, C\ge1$ such that,
for any probability Borel measure $\mu$ on $\Omega$ and
for any function class ${\bf F}\subset C(\Omega)$,
one has
$$
\varrho_{cN}({\bf F}, \|\cdot\|_{L^{\Phi_{p, \alpha, \beta}}(\mu)})
\le CN^{\frac{1}{2} - \frac{1}{p}}(\log 4N)^{\alpha/p}d_N({\bf F}, \|\cdot\|_\infty).
$$
In the case $\alpha = \beta= 0$ we recover the result of Theorem 20 in \cite{KPUU24}.
}
\end{example}

Let us emphasise that for the Orlicz norm generated by $\Phi_{2, \alpha, \beta}$
this estimate involves only the logarithmic factor $(\log 4N)^{\alpha/2}$
unlike in the case of the norms generated by $\Phi_{p, \alpha, \beta}$ with $p>2$, where an additional polynomial factor appears.
We can avoid such factor for {\it modified} sampling numbers
allowing certain polynomial oversampling (see Example \ref{ex7.2} below).
In more detail, first, combining Lemma \ref{lem7.1}, Theorem \ref{t5.1}, and Remarks \ref{rem5.1} and \ref{rem5.2},
we obtain the following result.

\begin{theorem}\label{t7.2}
Let $p\in (1, \infty)$.
Let also $\Phi$ be a ${\bf \Phi}$-function,
continuously differentiable on $(0, +\infty)$,
such that
$\Phi \in {\rm (Inc)}_p\cap{\rm (Dec)}$
and
$$
\sup\limits_{t>0}\bigr(\Phi(t)\Phi(t^{-1})\bigl)<\infty.
$$
There exist positive numbers $c:=c(\Phi, p)\ge1$ and $C:=C(\Phi, p)\ge 1$,
depending only on $\Phi$ and $p$, such that,
for any probability Borel measure $\mu$ on $\Omega$,
for any function class ${\bf F}\subset C(\Omega)$, and
for any
$$
m\ge c\Psi_\Phi\bigl(N^\frac{1}{\min\{p,2\}}\bigr)(\log 2N)^3,
$$
we have
$$
\varrho_{m}^*({\bf F}, \|\cdot\|_{L^\Phi(\mu)})
\le C\bigl(M_{\Phi, \Psi_\Phi}(m)\bigr)^{1/p}d_N({\bf F}, \|\cdot\|_\infty),
$$
where
$$
\Psi_\Phi(s):= \sup_{t>0}\frac{\Phi(ts)}{\Phi(t)}
$$
and
$$
M_{\Phi, \Psi_\Phi}(m):=
\max\Bigl\{1, \max\Bigl\{\frac{\Psi_\Phi(t)}{\Phi(t)}\colon t\in[\Phi^{-1}(1), \Phi^{-1}(m)]\Bigr\}\Bigr\}
\quad \forall m\in(0, +\infty).
$$
\end{theorem}

Second, we apply this theorem to the ${\bf \Phi}$-function $\Phi_{p, \alpha, \alpha}$ (see Example \ref{ex5.1}).

\begin{example}\label{ex7.2}
{\rm
Let $p\ge2$, $\alpha\ge0$, and $\Phi_{p, \alpha, \alpha}(t) := t^p\frac{(\ln(e+t))^\alpha}{(\ln(e+t^{-1}))^\alpha}$.
There exist positive constants $c:=c(p, \alpha), C:=C(p, \alpha)\ge1$, depending only on $p$ and $\alpha$, such that,
for any probability Borel measure $\mu$ on $\Omega$ and
for any function class ${\bf F}\subset C(\Omega)$,
we have
$$
\varrho_{cN^{p/2} (\log 2N)^{3+2\alpha}}^*({\bf F}, \|\cdot\|_{L^{\Phi_{p, \alpha, \alpha}}(\mu)})
\le C(\log 2N)^{\alpha/p}d_N({\bf F}, \|\cdot\|_\infty).
$$
}
\end{example}
For similar recovery results in $L^p$ (the case $\alpha = 0$), see Section {\bf R.3} in \cite{KKLT}.


\section*{Acknowledgements}

The authors would like to thank Mario Ullrich for sharing  the ideas that helped us to improve the original results in Theorem \ref{T-1sided}.

The first named author was supported by the Marie Sklodowska-Curie grant 101109701.
The second named author
was supported by PID2023-150984NB-I00, 2021 SGR 00087, AP 23488596.
The research was also supported by the Spanish State Research Agency, through the Severo Ochoa and Mar\'ia de Maeztu Program for Centers and
Units of Excellence in R\&D (CEX2020-001084-M). The authors thanks CERCA Programme (Generalitat de Catalunya) for institutional support.
The authors also would like to thank the Isaac Newton Institute for Mathematical Sciences, Cambridge,
for support and hospitality during the programme Discretization and recovery in high-dimensional spaces,
where work on this paper was partially undertaken. This work was supported by EPSRC grant EP/R014604/1.


\end{document}